\documentclass[twoside]{amsart}
\usepackage{latexsym}
\usepackage{amssymb}
\usepackage{amsmath}
\usepackage{amscd}
\usepackage{color}

\newtheorem{theorem}{Theorem}[section]
\newtheorem{lemma}[theorem]{Lemma}
\newtheorem{corollary}[theorem]{Corollary}
\newtheorem{proposition}[theorem]{Proposition}
\newtheorem{definition}[theorem]{Definition}
\newtheorem{example}[theorem]{Example}

\theoremstyle{remark}
\newtheorem*{remark}{Remark}

\theoremstyle{Main}
\newtheorem*{Main}{Theorem}

\def\ind#1#2{\left[{#1}:{#2}\right]}
\def\Aut#1{\text{Aut}({#1})}
\def\End#1{\text{End}({#1})}
\def\TA{{\bf TA}}
\def\TB{{\bf TB}}
\def\TBo{{\bf TB1}}
\def\TBt{{\bf TB2}}
\def\ident{e}
\def\triv{\{\ident\}}
\def\pbp#1{\overset{\rightarrow}{\sf par}({#1})}
\def\pbn#1{\overset{\leftarrow}{\sf par}({#1})}
\def\ml#1{{\sf lev}({#1})}
\def\ikk#1{{\sf bik}({#1})}
\def\nub#1{{\sf nub}({#1})}
\def\Sym{\text{\rm Sym}}
\def\BS{\text{\rm BS}}

\begin{document}
\bibliographystyle{plain}

\title[The scale for endomorphisms]{The Scale and Tidy Subgroups for Endomorphisms of Totally Disconnected Locally Compact Groups}

\begin{abstract}
The scale of an endomorphism, $\alpha$, of a totally disconnected, locally compact group $G$ is the minimum index $[\alpha(U) : \alpha(U)\cap U]$, for $U$ a compact, open subgroup of $G$. A structural characterization of  subgroups at which the minimum is attained is established. This characterization extends the notion of subgroup tidy for $\alpha$ from previously understood case when $\alpha$ is an automorphism to the case when $\alpha$ is merely an endomorphism. 
\end{abstract}

\author{George A. Willis}
\thanks{The support of A.R.C. grant DP0984342 is gratefully acknowledged}
\address{The University of Newcastle\\
University Drive\\
Callaghan, 2308\\
AUSTRALIA}

\maketitle

\section{Introduction}
\label{sec:introduction}

Automorphisms of totally disconnected, locally compact groups were investigated in \cite{Wi:structure} and \cite{wi:further}. The notion of a compact, open subgroup \emph{tidy} for the automorphism and the \emph{scale} of the automorphism were introduced in \cite{Wi:structure}, and it was shown in \cite{wi:further} that tidy subgroups and the scale characterize a certain minimizing property. These ideas have been used to answer questions concerning random walks and ergodic theory, \cite{DSW,JRW,PrevWu}, arithmetic groups \cite{ShWi}, and Galois groups~\cite{ChHr},  and have led to new developments in the structure theory of totally disconnected, locally compact groups. Throughout the paper, $G$ will denote a totally disconnected, locally compact group. 

The present paper extends the ideas of tidy subgroup and the scale to \emph{endomorphisms} of totally disconnected, locally compact groups, and shows that they still characterize when the index $[\alpha(U) : \alpha(U)\cap U]$ is minimized in this generality. Although the arguments follow along similar lines as for automorphisms, there are significant differences that require them to be completely reworked. One difference seen from the beginning is the definition of the group $U_+$, see Definition~\ref{defn:Uplus} below. Another is that, on occasions in the earlier papers a result is proved for the automorphism $\alpha$ and a `mirror' result deduced by applying it to $\alpha^{-1}$. In the case of endomorphisms, however, the `mirror' result must be formulated differently and requires a separate proof, compare Lemmas~\ref{lem:criterion} and~\ref{lem:criterion2} for example. Also, assertions that are immediate observations for automorphisms have a more restricted formulation and need careful proof for endomorphisms, see Propositions~\ref{prop:alphan} and~\ref{prop:tidy_subgroups}. In the time that has passed since the publication of~\cite{Wi:structure} and~\cite{wi:further} there have also been changes in terminology and notation and, it is hoped, improvements in the exposition that are given effect in this paper. 

Not all parts of~\cite{Wi:structure} and~\cite{wi:further} are superseded by the present paper. Examples of automorphisms used to illustrate the ideas are not repeated here and all examples below are of endomorphisms that highlight differences between the earlier ideas and their current version. (Further examples of the calculation of the scale and tidy subgroups for automorphisms of $p$-adic Lie groups may be found in~\cite{GlockScale1,GlockScale2}.) In addition, results in~\cite{Wi:structure} on continuity of the scale for inner automorphisms and in~\cite{wi:further} on stability properties of the scale when passing to subgroups and quotient groups do not make sense for endomorphisms and are not reproduced here. 

Endomorphisms of locally compact totally disconnected groups have also been considered in~\cite{BDG-B}, and of compact totally disconnected groups in~\cite{reid}. The results in the latter paper complement those given here: when $G$ is compact, the index $[\alpha(U) : \alpha(U)\cap U]$ is minimised by $G$ itself and the scale of every endomorphism is thus equal to~$1$ while, on the other hand, the present paper has most to say about endomorphisms of non-compact groups having scale greater than~$1$. There is an extensive study of endomorphisms of abstract groups (or, from the present point of view, of groups with the discrete topology) in the literature, see~\cite{BaumSol,Kochloukova,PNeumann,nonHopf} for example. The focus in these papers is often on Hopfian (all onto endomorphisms are automorphisms) and co-Hopfian (all one-to-one endomorphisms are automorphisms) groups. The last section of the present paper presents examples of Hopfian and co-Hopfian non-discrete groups. 

An analogy between automorphisms of totally disconnected, locally compact groups and linear transformations has motivated development of the theory of the scale and tidy subgroups, see~\cite{Wi:SimTr}. This analogy suggests the desirability of extending the theory to endomorphisms but the author did not succeed in doing so hitherto. However, being asked that question by Yves Cornulier during the conference `Locally compact groups beyond Lie theory' held at Spa in early April 2013 led to a new and this time successful attempt. I am grateful to Yves for his question and to the organisers of the conference for giving us the opportunity for discussion. 

\section{Outline of the main theorem and its proof}
\label{sec:outline}

The main definitions and ideas concerning endomorphisms are presented in this section along with an outline of the paper. Detailed arguments are presented in the following sections. The starting point is the following fundamental theorem about totally disconnected, locally compact groups, which was proved by van Dantzig in the 1930's, see~\cite{Dant}, \cite[Theorem~II.2.3]{MontZip} or~\cite[Theorem~II.7.7]{HewRoss}. 

\begin{theorem}[van Dantzig]
\label{thm:COSubgroups}
Let $G$ be a totally disconnected locally compact group. Then $G$ has a base of neighbourhoods of the identity consisting of compact, open subgroups.
\end{theorem}

An \emph{endomorphism} of $G$ is a continuous homomorphism $\alpha: G\to G$, and an \emph{automorphism} is an endomorphism that is one-to-one and onto having continuous inverse. The set of endomorphisms of $G$ forms a semigroup under composition, which will be denoted by $\End{G}$, and the automorphisms form a group, which will be denoted by $\Aut{G}$. 

Endomorphisms of $G$ are studied here through their action on the compact, open subgroups of $G$. If $U$ is such a subgroup, then $\alpha(U)$ is a compact subgroup of $G$ that is not necessarily open. The intersection $\alpha(U)\cap U$ is an open subgroup of $\alpha(U)$ and so  has finite index $[\alpha(U) : \alpha(U)\cap U]$, which it will be referred to as the \emph{displacement index} of the subgroup $U$ by $\alpha$. The displacement index equals~1 if and only if $U$ is invariant under~$\alpha$ and the following definition therefore gives a measure of the extent by which $\alpha$ fails to leave any compact, open subgroup invariant.

\begin{definition}
\label{def:ScaleTidy} 
Let $G$ be a totally disconnected locally compact group and let
$\alpha$ be an endomorphism of ${G}$. The \emph{scale} of $\alpha$ is the positive integer
$$
s(\alpha) := \min \left\{\ind{\alpha(U)}{\alpha(U)\cap U} \mid U\leq G\text{
is compact and open}\right\}.
$$
The compact open subgroup $U$ is \emph{minimizing} for $\alpha$ if the minimum
is attained at~$U$. 
\end{definition}

It is shown in~\cite{wi:further} that a compact, open subgroup is minimizing for an automorphism $\alpha$ if and only if it satisfies structural conditions called the \emph{tidiness} criteria for $\alpha$. The main result of this paper is the following, which extends this characterization to endomorphisms, with modified tidiness criteria. 
\begin{Main}[The Structure of Minimizing Subgroups]
\label{thm:tidystucture}
Let $\alpha$ be an endomorphism of the totally disconnected, locally compact group ${G}$. For a compact, open subgroup, $U$, of $G$ put 
\begin{align*}
U_+ &= \left\{ x\in U \mid \exists \{x_n\}_{n\in\mathbb{N}}\subset U \text{ with }x_0 = x\text { and }\alpha(x_{n+1}) = x_n\text{ for every }n\right\}\\
\text{ and } U_- &= \left\{ x\in U \mid \alpha^n(x)\in U \text{ for every }n\in\mathbb{N}\right\}.
\end{align*}
Then $U$ is minimising for $\alpha$ if and only if:
\begin{description}
\item[\TA] $U = U_+U_-$;
\item[\TB1] $U_{++} = \bigcup_{n\geq0} \alpha^n(U_+) $ is closed; and 
\item[\TB2] the sequence of integers $\left\{[\alpha^{n+1}(U_+) : \alpha^n(U_+)]\right\}_{n\in\mathbb{N}}$ is constant.
\end{description}
\end{Main}
It is immediate from the definitions that $U_+$, $U_-$ and $U_{++}$ are subgroups of $G$ and it will be seen that $U_+$ and $U_-$ are closed. 
\begin{definition}
\label{defn:TIDY}
Let $G$ be a totally disconnected, locally compact group and $\alpha$ be in $\End{G}$. The compact, open subgroup $U$ is \emph{tidy above} for $\alpha$ if it satisfies \TA\ and \emph{tidy below} if it satisfies \TB{\bf1} and \TB{\bf2}. A subgroup that is tidy above and below is \emph{tidy} for $\alpha$. 
\end{definition}

Comparing with the definitions and results for automorphisms in~\cite{Wi:structure}, condition \TB{\bf2} was not required there for tidiness below because it is automatic for one-to-one endomorphisms. In~\cite{Wi:structure} it was shown as well that an automorphism is tidy below if and only if $V_{--}$ is closed. The same criterion continues to hold for endomorphisms if $V_{--}$ is defined to be $\bigcup_{n\geq0} \alpha^{-n}(U_-)$, as is shown in Proposition~\ref{prop:inV} below. However, although more complicated, it is more illuminating to state criteria for tidiness below in terms of $V_{++}$ because these are the criteria that emerge naturally in the course of the proof. It will be seen also that the definition of the subgroup~$U_+$ given in the main theorem is equivalent to the definition, $U_+ = \bigcap_{n\geq0}\alpha^n(U)$, given in~\cite{Wi:structure} when $\alpha$ is an automorphism but that the more complicated definition is necessary when $\alpha$ is merely an endomorphism.

The theorem characterizing the structure of minimizing subgroups is a consequence of the \emph{tidying procedure} and is proved in Sections~\ref{sec:factoring}--\ref{sec:Min_equiv_Tidy} as this procedure is developed. Given an endomorphism $\alpha$ of $G$, the tidying procedure takes any compact
open subgroup $U$ of $G$ and modifies it in two steps to produce a subgroup tidy for $\alpha$. The first step, carried out in Section~\ref{sec:factoring}, results in a subgroup that is tidy above and the second, carried out in Sections~\ref{sec:V0} and~\ref{sec:tidy_below}, in a subgroup that is tidy below as well. That the subgroup produced in Section~\ref{sec:tidy_below} in fact satisfies the criteria \TBo\ and \TBt\ is shown in Section~\ref{sec:TB_criteria}. Each step in the tidying procedure reduces the displacement index and this, together with the fact established in
Section~\ref{sec:Min_equiv_Tidy} that all tidy subgroups have the same displacement index, shows the equivalence between tidiness for $\alpha$ and being minimizing for $\alpha$. In Section~\ref{sec:Moller}, a limit formula for the scale of $\alpha$ established for automorphisms by R.\,G.\,M\"oller in \cite[Theorem~7.7]{Moller} is extended to endomorphisms.  Certain subgroups of $G$ associated with $\alpha$, such as the parabolic and Levi subgroups, are studied in Section~\ref{sec:subgroups}.

The orbit $\{\alpha^n(x)\}_{n\in\mathbb{Z}}$ plays a role at many points in \cite{Wi:structure} and \cite{wi:further}, where only the case when $\alpha$ is an automorphism is considered. When~$\alpha$ does not have an inverse only the orbit  $\{\alpha^n(x)\}_{n\in\mathbb{N}}$ makes sense and the following concept fills the role of the orbit under negative powers of $\alpha$ in that case.
\begin{definition}
\label{defn:regressive}
Let $\alpha\in\End{G}$ and $x\in G$. An \emph{$\alpha$-regressive sequence} for $x$ is a sequence $\{x_{n}\}_{n\in\mathbb{N}}$ such that $x_0 = x$ and $\alpha(x_{n+1}) = x_{n}$ for each $n\in\mathbb{N}$. 
\end{definition}
An $\alpha$-regressive sequence for $x$ need not exist, and when it does exist it need not be unique. It will often be a hypothesis that an $\alpha$-regressive sequence exists and either has an accumulation point or is bounded, that is, has compact closure. 

The tidying procedure will be illustrated with examples in addition to those that may be found in~\cite{Wi:structure,wi:further}. In many of these examples the group $G$ will be either the additive group of the field $F_p((t))$ of formal Laurent series over the field $F_p$ of order~$p$, or of its subring $F_p[[t]]$ of formal power series. Elements of the group will be denoted $f = \sum_{n\geq N} f_nt^n$ for some $N\in\mathbb{Z}$. The subring $F_p[[t]]$ is isomorphic to $F_p^{\mathbb{N}}$ and the field $F_p((t))$ is locally compact under the topology in which $F_p[[t]]$ is an open subgroup homeomorphic to the compact space $F_p^{\mathbb{N}}$ with the product topology. 

\section{Subgroups that are tidy above}
\label{sec:factoring}

The first step towards finding a  subgroup tidy for $\alpha$ in $\End{G}$ is to show that  each compact open subgroup, $U$, of $G$ contains a subgroup, $V$, with finite index in~$U$ which factors into a subgroup $V_+$ that $\alpha$ expands and a subgroup $V_-$ that~$\alpha$ shrinks. In other words, that $U$ has a open subgroup $V$ that is tidy above for~$\alpha$. This smaller compact open subgroup has smaller displacement index than the original with equality if and only if $U$ is already tidy above. 

The special case when $G$ is discrete  motivated the argument given in \cite{Wi:structure} and this special case also highlights that the argument must be modified in order to extend it to cover endomorphisms. When $G$ is discrete, its compact subgroups are finite and for any such subgroup, $F$ say, the sequence $\left\{\bigcap_{k=0}^n\alpha^k(F)\right\}_{n\geq0}$ stabilises eventually. If $\alpha$ is an automorphism, the limit of the sequence is $\alpha$-invariant and is therefore a minimizing subgroup. However, that need no longer be so when $\alpha$ is merely an endomorphism, as the next example shows. In the example, $C_p$ denotes the cyclic group of order $p$ and $\bar{0}$ denotes its identity. 
\begin{example}
\label{ex:not_stable}
Let $G = C_p^3 = \left\{ (g_1,g_2,g_3)\mid g_j\in C_p\right\}$, and let $\alpha\in\End{G}$ be given by 
$$
\alpha(g_1,g_2,g_3) = (\bar{0},g_1+g_3,g_2).
$$
Consider the subgroup $F = \left\{(g_1,g_2,\bar{0})\mid g_j\in C_p\right\}$ of $G$. Then 
$$
\alpha^k(F) =  \left\{(\bar{0},g_2,g_3)\mid g_j\in C_p\right\}\text{ for every }k\geq1
$$ 
and 
$$
F_+ := \bigcap_{k=0}^n \alpha^k(F)  = \left\{(\bar{0},g_2,\bar{0})\mid g_j\in C_p\right\}\text{ for every }n\geq1.
$$ 
However $\alpha(F_+) = \left\{(\bar{0},\bar{0},g_3)\mid g_j\in C_p\right\}$ and $F_+$ is not stable under $\alpha$. Note too that $\bigcap_{k=0}^\infty \alpha(\alpha^k(F))  = \left\{(\bar{0},g_2,g_3)\mid g_j\in C_p\right\}$ is strictly greater than $\alpha\left(F_+)\right)$. 
\end{example}

The automorphism $\alpha$ in the example is not one-to-one and for such maps the inclusion $\alpha(A\cap B) \leq \alpha(A)\cap \alpha(B)$ may be strict. It is for this reason that $F_+$ is not $\alpha$-stable. For the same reason, defining $U_+$ as in~\cite{Wi:structure} does not lead to the  conclusion that $\alpha(U_+)\geq U_+$. The following modified definition is equivalent to the original one when $\alpha$ is an automorphism. 
\begin{definition}
\label{defn:Uplus}
Let $U$ be a compact open subgroup of $G$ and $\alpha$ be an endomorphism of $G$. Define subgroups $U_n$, $n\geq0$ and $U_+$ of $U$ recursively by setting $U_0 = U$ and  
\begin{equation}
\label{eq:Uplus}
U_{n+1} = U\cap \alpha(U_n) \text{ for } n\geq 0 \text{ and } U_+ = \displaystyle \bigcap_{n=0}^\infty U_n.
\end{equation}
\end{definition}
Since $U$ is compact and $\alpha$ continuous, $U_n$ and $U_+$ are compact subgroups of~$U$ and an induction argument shows that $\{U_n\}_{n\geq0}$ is a non-increasing sequence. Unlike the situation in \cite{Wi:structure}, $U_n$ is not open unless $\alpha$ is an open map. The subgroups $U_n$ and $U_+$ have the alternative characterisations:
\begin{align}
\label{eq:Uplus2}
U_n &= \left\{ u\in U \mid \exists v\in U\text{ with } \alpha^j(v)\in U \text{ for }j\in\{0,1,\dots, n\}  \text{ and }  u = \alpha^n(v)\right\};\\
U_+ &= \left\{ u\in U \mid \exists u \text{ has an $\alpha$-regressive sequence }\{u_n\}_{n\in\mathbb{N}}\subset U\right\}.
\notag
\end{align}

When defined in this way, $U_+$ is expanded by $\alpha$ as required. 
\begin{lemma}
\label{lem:Uplus}
Let $\{K_n\}_{n\geq0}$ be a non-increasing sequence of compact subsets of $G$ and put $K_+ = \bigcap_{n\geq0} K_n$. Then $\alpha(K_+) = \bigcap_{n\geq0} \alpha(K_n)$. 

Let $U_+$ be as in Definition~\ref{defn:Uplus}. Then
\begin{equation}
\label{eq:Uplus}
\alpha(U_+)\geq U_+\text{ and }  U\cap\alpha(U_+) = U_+.
\end{equation}
\end{lemma} 
\begin{proof}
Since $K_+\subseteq K_n$ for each $n$, $\alpha(K_+) \subseteq \bigcap_{n\geq0} \alpha(K_n)$. That the sets $K_n$ are compact is important for the proof of the reverse inclusion. Consider $y$ in $\bigcap_{n\geq0} \alpha(K_n)$. For each $n$, the set $X_n =\left\{ x\in K_n\mid \alpha(x) = y\right\}$ is non-empty and compact, and the sequence $\{X_n\}_{n\in\mathbb{N}}$ is non-increasing because $\{K_n\}_{n\in\mathbb{N}}$ is. Hence $\bigcap_{n\in\mathbb{N}} X_n$ is not empty and any $x$ in this intersection belongs to $K_+$ and satisfies $y=\alpha(x)$. 

To see (\ref{eq:Uplus}), note that $\alpha(U_n)\geq U_{n+1}$ by definition. Hence
$$
\alpha(U_+) = \bigcap_{n\geq0} \alpha(U_n) \geq \bigcap_{n\geq1} U_n = U_+.
$$
For the last claim: $U\cap \alpha(U_+) = \bigcap_{n\geq0} U\cap\alpha(U_n) = \bigcap_{n\geq0} U_{n+1} = U_+$. 
\end{proof}
For $G$ discrete, $U_+$ is finite and the above argument shows that $U_+$ is $\alpha$-stable. 

In~\cite{Wi:structure} the subgroup $U_-$ is defined by replacing $\alpha$ with $\alpha^{-1}$ in the definition of~$U_+$. That cannot be done when $\alpha$ is an endomorphism and so this definition too must be modified. (Note that $\alpha^{-k}(U)$ denotes the set-theoretic inverse image of $U$ under $\alpha^k$.) 

\begin{definition}
\label{defn:Uminus}
Let $U$ be a compact open subgroup of $G$ and $\alpha$ be an endomorphism of $G$. Define subgroups $U_{-n}$, $n\geq0$ and $U_-$ of $U$ by 
\begin{align}
\label{eq:Uminus}
U_{-n} &= \bigcap_{k=0}^n \alpha^{-k}(U) = \left\{ u\in U \mid \alpha^k(u)\in U \text{ for } k\in\{0,\dots, n\}\right\}, \text{ and }\\
U_- &= \displaystyle \bigcap_{n=0}^\infty U_{-n} = \left\{ u\in U \mid \alpha^k(u)\in U \text{ for all } k\in\mathbb{N}\right\}.\notag
\end{align}
\end{definition}
The definition of $U_{-n}$, unlike that of $U_n$, is not recursive and does not need to be: making a recursive definition of $U_{-n}$ analogous to that of $U_n$ yields the same subgroup as in~(\ref{eq:Uminus}).  Since $\alpha$ is a continuous homomorphism, $U_{-n}$ is an open subgroup of $U$ for each $n\geq0$ and $U_-$ is a closed subgroup. 

The identity $
\alpha^k\left(\bigcap_{j=m}^n \alpha^j(U)\right) = \bigcap_{j=m+k}^{n+k} \alpha^j(U)
$ is used frequently and without remark in arguments when $\alpha$ is an automorphism but need not be true or even make sense when $\alpha$ is an endomorphism. Part~(\ref{eq:Uminus1}) of the next lemma is a reformulation of this identity that holds for endomorphisms. 
\begin{lemma}
\label{lem:Uminus}
The sequence $\left\{U_{-n}\right\}_{n\geq0}$ of open subgroups of $U$ is non-increasing. These subgroups satisfy 
\begin{enumerate}
\item \label{eq:Uminus0}
$U_{-N-m} = (U_{-N})_{-m}\text{ for all }m,N\geq0$,
\item
\label{eq:Uminus1}
$\alpha^k(U_{-n}) = 
\begin{cases}
U_k\cap U_{k-n}, & \text{ if } 0\leq k\leq n\\
\alpha^{k-n}(U_n), & \text{ if } k\geq n
\end{cases}$, \quad
and 
\item \label{eq:Uminus2}
$(U_{-n})_j = U_j \cap U_{-n}  \text{ for } j\geq0,\ \text{ and }\ (U_{-n})_+ =  U_+ \cap U_{-n}$.
\end{enumerate}
\end{lemma}
\begin{proof}
It is immediate from the definition that $U_{-n-1}\leq U_{-n}$ for each $n\geq0$ and so the sequence in non-increasing. To see (\ref{eq:Uminus0}), consider that $u$ belongs to $U_{-N-m}$ if and only if $\alpha^j(u)\in U$ for every $j\in\{0,1,\dots, m+N\}$. This is equivalent to $\alpha^k(\alpha^j(u))$ belonging to $U$ for every $j\in\{0,1,\dots, m\}$ and $k\in\{0,1,\dots,N\}$, and then to $\alpha^j(u)$ belonging to $U_{-N}$ for every $j\in\{0,1,\dots, m\}$.  

For (\ref{eq:Uminus1}), consider first the case when $0\leq k\leq n$. Then
$$
\alpha^k(U_{-n}) = \left\{ \alpha^k(u) \mid \alpha^j(u)\in U \text{ for } 0\leq j\leq n\right\}.
$$ 
Set $\alpha^k(u)$ equal to $v$. Then~$v$ belongs to $U_k$ because $\alpha^j(u)\in U$ for $j\in\{0,1,\dots k\}$. Furthermore,  for $j$ in $\{k,\dots n\}$, $\alpha^{j-k}(v) = \alpha^j(u)$, which is in $U$, whence~$v$ belongs to $U_{k-n}$. Hence $\alpha^k(U_{-n}) \leq U_k\cap U_{k-n}$. Similarly, if $v$ is in $U_k$, then $v = \alpha^k(u)$ where $\alpha^j(u)\in U$ for all $j\in\{0,1,\dots, k\}$ and, if $v$ also belongs to $U_{k-n}$, then $\alpha^{j-k}(v) = \alpha^j(u)\in U$ for all   $j\in \{k,\dots,n\}$, whence $v = \alpha^k(u)$ where $u$ is in $U_{-n}$. For the case when $k\geq n$, note that it has already been shown that $\alpha^n(U_{-n}) = U_n$. Hence $\alpha^k(U_{-n}) = \alpha^{k-n}(U_n)$ for $k\geq n$.
 
The first claim in (\ref{eq:Uminus2}) is proved by induction on $j$, and certainly holds when $j=0$. Note also that proof is only required for $n\geq1$. Assume that it has been shown for some $j$ that $(U_{-n})_j = U_j \cap U_{-n}$. Then
$$
(U_{-n})_{j+1} = U_{-n} \cap \alpha(U_j\cap U_{-n}). 
$$
Applying~(\ref{eq:Uminus1}) twice shows that $\alpha(U_j\cap U_{-n}) = \alpha^{j+1}(U_{-n-j}) = U_{j+1}\cap U_{1-n}$. Hence $(U_{-n})_{j+1} = U_{-n} \cap U_{j+1}$ and the induction continues. That $(U_{-n})_+ = U_+ \cap U_{-n}$ then follows immediately from the definition of $U_+$.
\end{proof}

The special case of (\ref{eq:Uminus1}) when $k=n$, asserting that $\alpha^n(U_{-n}) = U_n$, that has already been highlighted in the above proof will be used most frequently. 

When $G$ is discrete the non-increasing sequence $\{U_n\}_{n\in\mathbb{N}}$ of finite groups eventually stabilizes. The version of this for non-discrete $G$ is that the sequence $[\alpha(U_n) : \alpha(U_n)\cap U_n]$ of positive integers is non-increasing, as shown in the next lemma, and stabilizes at a subgroup~$U_n$ that will subsequently be shown to be tidy above for $\alpha$.
\begin{lemma}
\label{lem:stable}
Let $\alpha\in \End{G}$ and let $U$ be a compact open subgroup of $G$.  Then: 
\begin{enumerate}
 \item $[\alpha(U) : \alpha(U)\cap U] = [U:U_{-1}]$;
 \label{lem:stable1}
 \item $[U_j\cap U_{-l} : U_j\cap U_{-m}] = [U_{-j-l} : U_{-j-m}]$ for all non-negative integers $j$ and $l\leq m$, in particular, $[U_j : U_j\cap U_{-1}] = [U_{-j} : U_{-j-1}]$;
  \label{lem:stable2}
 \item the sequence $\{[U_{j} : U_{j}\cap U_{-1}]\}_{j\geq0}$ is non-increasing; and 
   \label{lem:stable3}
\item it stabilises at the value $[U_+ : U_+\cap U_{-1}]$ when $j$ is such that $U_j\subseteq U_+U_{-1}$.
  \label{lem:stable4}
\end{enumerate}
\end{lemma}
\begin{proof}
(\ref{lem:stable1}) The map $u(U\cap\alpha^{-1}(U)) \mapsto \alpha(u)(\alpha(U)\cap U)$ is well-defined from $U/U_{-1}$ to $\alpha(U)/(\alpha(U)\cap U)$. It clearly maps onto $\alpha(U)/(\alpha(U)\cap U)$ and is one-to-one because, if  $\alpha(u)(\alpha(U)\cap U) = \alpha(v)(\alpha(U)\cap U)$, then $\alpha(v^{-1}u)\in U$, which implies that $v^{-1}u\in U_{-1}$. 

(\ref{lem:stable2}) Lemma~\ref{lem:Uminus}(\ref{eq:Uminus1}) implies that 
$$
[U_j\cap U_{-l} : U_j\cap U_{-m}] = [\alpha^j(U_{-j-l}): \alpha^j(U_{-j-m})].
$$ 
The claim results by verifying that the map $uU_{-j-m} \mapsto \alpha^j(u)(\alpha^j(U_{-j-m}))$ is a well-defined bijection from $U_{-j-l}/U_{-j-m}$ to $\alpha^j(U_{-j-l})/\alpha^j(U_{-j-m})$. Only the claim that it is one-to-one requires any argument. For this, note that if $u\in U_{-j-l}$ and $\alpha^j(u)\in U_{-j-m}$, then $u =  vz$ for some $v\in U_{-j-m}$ and $z\in U_{-j-l}\cap\ker(\alpha)$. Since $z\in U$ and $\alpha(z) = \ident$, it follows that $z\in U_-$ and $u\in U_{-j-m}$. 

(\ref{lem:stable3}) \& (\ref{lem:stable4}) Define a map $U_{j+1}/(U_{j+1}\cap U_{-1}) \to U_j/(U_j\cap U_{-1})$ by 
$$
u(U_{j+1}\cap U_{-1}) \mapsto u(U_j\cap U_{-1}), \ \ (u\in U_{j+1}).
$$
This map is well-defined and one-to-one, which implies that $[U_{j+1} : U_{j+1}\cap U_{-1}]$ is less than or equal to $[U_j : U_j\cap U_{-1}]$. The similarly defined map from $U_{+}/(U_{+}\cap U_{-1})$ to $U_j/(U_j\cap U_{-1})$ is also well-defined and one-to-one, and is onto if and only if $U_j\subseteq U_+U_{-1}$.  Hence the sequence $\{[U_{j} : U_{j}\cap U_{-1}]\}_{j\geq0}$ stabilizes as soon as $U_j\subseteq U_+U_{-1}$ and the value that it stabilizes at is $[U_{+} : U_{+}\cap U_{-1}]$. 
\end{proof}
\begin{remark}
The conclusion of Lemma~\ref{lem:stable}(\ref{lem:stable4}) that there is a $j\geq0$ such that $U_j$ is contained in $U_+U_{-1}$ also follows by compactness, since the sequence of compact subgroups $\{U_j\}_{j\geq0}$ is decreasing with intersection $U_+$ and because $U_+U_{-1}$ is an open neighbourhood of $U_+$. 
\end{remark}

The next few results establish that a subgroup $U_{-j}$, where $j$ as found in Lemma~\ref{lem:stable} is such that $[U_j : U_j\cap U_{-1}] = [U_+ : U_+\cap U_{-1}]$, is tidy above for $\alpha$.  
\begin{lemma}
\label{lem:newfindN}
Let $N$ be a non-negative  integer such that $U_N\subseteq U_+U_{-1}$. Then
$$
U_{-n} = (U_+\cap U_{-n})U_{-n-1} \text{ for every }n\geq N.
$$
Put $V = U_{-N}$. Then $V$ is an open subgroup of $U$ that satisfies
\begin{equation*}
\label{eq:minus_N}
V_{-n} = (V_{-n})_+V_{-n-1} \text{ for all }n\geq0.
\end{equation*}
\end{lemma}
\begin{proof}
If $n\geq N$, then $U_n\subseteq U_+U_{-1}$. Consider $v\in U_{-n}$. Then $\alpha^n(v)$ belongs to $U_n$ and so is equal to $yw$ for some $y\in U_+$ and $w\in U_{-1}$. By definition of $U_+$, $y = \alpha^n(y')$ for some $y'\in U_+\cap U_{-n}$. Furthermore, since $U_n$ is a group that contains $U_+$,  it follows that $w\in U_n\cap U_{-1}$ and thence, by Lemma~\ref{lem:Uminus}(\ref{eq:Uminus1}), that there is $w'\in U_{-n-1}$ such that $w = \alpha^n(w')$. Hence 
$$
v = y'w'z\text{ for some }z\in \ker(\alpha^n).
$$
The element $z$ belongs to $U_{-n}$ because $U_{-n}$ is a group and $v$, $y'$ and $w'$ do. Since $\alpha^{n+1}(z) = \ident$, it is in fact the case that $z\in U_{-n-1}$. Then, $w'z\in U_{-n-1}$ and it has been shown that $U_{-n}\subseteq (U_+\cap U_{-n})U_{-n-1}$. The reverse inclusion holds because $U_+\cap U_{-n}$ and $U_{-n-1}$ are both contained in $U_{-n}$. 

Let $V = U_{-N}$. Then it follows immediately from Lemma~\ref{lem:Uminus}(\ref{eq:Uminus0}) and the first part that 
$$
V_{-n} = (V_+\cap V_{-n})V_{-n-1} \text{ for every }n\geq0.
$$
Applying (\ref{eq:Uminus2}) then yields the claim.
\end{proof}

\begin{proposition}
\label{prop:equivalents}
Let $\alpha\in\End{G}$ and $U$ be a compact, open subgroup of $G$. Then the following are equivalent.
\begin{enumerate}
\item $U = U_+U_{-1}$.
\label{prop:equivalents1}
\item $U_{-n} = (U_{-n})_+U_{-n-1}$ for all $n\geq0$. 
\label{prop:equivalents2}
\item $U = U_+U_{-n}$ for all $n\geq0$.
\label{prop:equivalents3}
\item $U = U_+U_-$.
\label{prop:equivalents4}
\end{enumerate}
\end{proposition}
\begin{proof}
That (\ref{prop:equivalents1}) implies (\ref{prop:equivalents2}) is Lemma~\ref{lem:newfindN}. 

Assume that (\ref{prop:equivalents2}) holds. 
Then in particular $U = U_+ U_{-1}$ and induction on $n$, using the fact that $(U_{-n})_+\leq U_+$, implies (\ref{prop:equivalents3}). 

Next assume that (\ref{prop:equivalents3}) holds and let $u\in U$. Then $uU_{-n} \cap U_+$ is a non-empty  compact set for each $n\geq0$ and $uU_{-n-1} \cap U_+\leq uU_{-n} \cap U_+$. Hence 
$$
\bigcap_{n\geq0} uU_{-n} \cap U_+ \ne \emptyset.
$$ 
Let $x$ be in this intersection, so that $x\in U_+$ in particular. In addition, $x^{-1}u\in U_{-n}$ for every $n\geq0$, whence $x^{-1}u\in U_-$ and (\ref{prop:equivalents4}) is established.

That (\ref{prop:equivalents4}) implies (\ref{prop:equivalents1}) is clear.
\end{proof}

\begin{remark}  Note that since
$U$, $U_+$ and $U_-$ are groups, they are closed under taking inverses. Since moreover 
$(u_+u_-)^{-1} = u_-^{-1}u_+^{-1}$, it follows that $U$ is tidy above for the endomorphism $\alpha$ if and only if $U = U_-U_+$. 
\end{remark}

The following is an immediate consequence of Lemma~\ref{lem:stable} and Proposition~\ref{prop:equivalents}.
\begin{proposition}
\label{prop:propertyTA}
Let $G$ be a totally disconnnected locally compact group,  let
$\alpha$ be in $\End{G}$ and let $U$ be a compact open subgroup of $G$. Then there is an integer $N$  such that
$V := U_{-N}$ is tidy above for $\alpha$. 
For this $V$,  
$$
\left[\alpha(V) : V\cap \alpha(V)\right] \leq 
\left[\alpha(U) : U\cap \alpha(U)\right]
$$ 
with equality if and
only if $U$ satisfies~\TA\ for $\alpha$.
\endproof
\end{proposition}

\begin{lemma}  
\label{lem:indexreduced} 
Let $\alpha\in \End{G}$ and let $U$ be a compact open subgroup of $G$.  Then 
$$
\left[ \alpha
(U_+) : U_+ \right] \leq \left[ \alpha(U) : \alpha (U) \cap U \right ]
$$ 
with equality if and only if $U$ is tidy above for $\alpha$.
\end{lemma}

\begin{proof}  Define a map $\psi : \alpha (U_+)/U_+ \rightarrow 
\alpha (U) / (\alpha (U) \cap U )$ by 
$$
\psi(uU_+) = u \left( \alpha (U) \cap U \right) \text{ for }u\in \alpha
(U_+).
$$
 Then $\psi$ is well-defined because
$U_+\subset \alpha (U) \cap U$. To see that $\psi$ is injective, let
$u,v\in \alpha(U_+)$ be such that   $\psi (uU_+) = \psi (vU_+)$. Then
$u(\alpha(U)\cap U)    = v(\alpha(U)\cap U)$, whence 
$v^{-1}u \in \alpha(U)\cap U$. 
Since $u$ and $v$ belong to $\alpha(U_+)$, it follows that 
$$
	v^{-1}u \in  U\cap \alpha(U_+) = U_+, \text{ by Lemma~\ref{lem:Uplus}}.  
$$
Hence $uU_+= vU_+$ and $\psi$ is injective.

Since $\psi$ is injective,  $\left [ \alpha (U_+) : U_+
\right ] \leq  \left [\alpha (U) : \alpha (U) \cap U \right ]$. Equality occurs if and only if $\psi$ is onto, which is the case if and only if $\alpha(U) = \alpha (U_+)(\alpha (U) \cap U)$. Suppose that $\alpha(U) = \alpha (U_+)(\alpha (U) \cap U)$ and consider $u\in U$. Then there are $v\in U_+$ and $w\in \alpha(u)\cap U$ such that $\alpha(u) = \alpha(v)w$. There is $w'\in U_{-1}$ such that $w = \alpha(w')$, whence $\alpha(u) = \alpha(v)\alpha(w')$ and $ u = vw'z$ for some $z\in \ker(\alpha)$. Since $z$ belongs to $U$ and $\alpha(z) = \ident$, it follows that $z\in U_{-1}$ and we have shown that $U = U_+U_{-1}$. Then $U$ is tidy above for $\alpha$ by Proposition~\ref{prop:equivalents}. The reverse argument is clear. 
\end{proof}

Examples illustrating Proposition~\ref{prop:propertyTA} for various automorphisms $\alpha$ may be found in~\cite{Wi:structure}. The procedure for finding a subgroup tidy for $\alpha$ will be carried through in the present paper for the following example, in which $\alpha$ is an endomorphism that is neither one-to-one nor onto. 
\begin{example}
\label{ex:compact_example}
Let $G$ be the additive group of the topological ring $F_p[[t]]$, which is compact. For $g = \sum_{n\geq0} g_nt^n$, define $\alpha(g)$ by
$$
\alpha(g)_n =
\begin{cases}
g_{n+1}+g_{n+2}, & \text{ if $n$ is even}\\
\bar{0}, & \text{ if $n$ is odd}
\end{cases}.
$$
Then $\alpha\in\End{G}$. Let $U$ be the open subgroup of $G$
$$
U = \left\{ g\in G \mid g_2 = \bar{0}\text{ and } g_6 = g_{8}\right\}.
$$
Then calculation shows that 
\begin{align*}
\alpha(U) &= \left\{ g\in G \mid g_n = \bar{0}\text{ if $n$ is odd}\right\},\\
U_1 &= \left\{ g\in G \mid g_n=\bar{0}\text{ if $n$ is odd or equals }2 \text{ and }g_6=g_{8}\right\},\\
U_2 &= \left\{ g\in G \mid g_n=\bar{0}\text{ if $n$ is odd or equals }0\text{ or }2 \text{ and }g_4 = g_6=g_{8}\right\},\\
\text{and }\ \  U_3 &= U_+ = \left\{ g\in G \mid g_n=\bar{0}\text{ if $n$ is odd or less than }9\right\}.\\
{Similarly,}
U_{-1} &= \left\{ g\in G \mid g_2 = \bar{0} = g_3+g_4 \text{ and } g_6 = g_{8},\ g_7+g_{8} = g_{9}+g_{10}\right\},\\
U_{-2} &= \left\{ g\in G \mid g_2 = \bar{0} = g_3+g_4 = g_5+g_6\text{ and } g_6 = g_{8},\ g_7+g_{8} = g_{9}+g_{10} \right.\\
& \phantom{= } \qquad \qquad\  \left. = g_{11}+g_{12}  \right\},\\
U_{-3} &= \left\{ g\in G \mid g_2 = \bar{0} = g_3+g_4 = g_5 + g_6 = g_7+g_{8} = g_{9}+g_{10} = g_{11}+g_{12} \right.\\
& \phantom{=} \qquad \qquad\ \left.  = g_{13}+g_{14}  \text{ and } g_6 = g_{8} \right\},\\
\text{and }\ \ U_- &= \left\{ g\in G \mid  g_2  = \bar{0} = g_{2j+1} + g_{2j+2} \text{ for all }j\geq 1  \text{ and } g_6 = g_{8}\right\}.
\end{align*}
Hence $[\alpha(U) : \alpha(U)\cap U] = p^2$ and $U_2\not\leq U_+U_{-1}$ but $U_3\leq U_+U_{-1}$. Note too that $\alpha(U_-) = \left\{ g\in G \mid g_n = \bar{0}\text{ if } n\geq 1\right\}$ and $\alpha^2(U_-) = \triv$.

The subgroup $V = U_{-3}$ is tidy above for $\alpha$. We have $V_- = U_-$ and 
$$
V_+ = \left\{ g\in G \mid g_n = \bar{0} \text{ if } n\leq 15 \text{ or is odd}\right\},
$$
so that $V = V_+V_-$, and $[\alpha(V) : \alpha(V)\cap V] = p$. 
\end{example}

\section{A compact $\alpha$-stable subgroup that contains bounded $\alpha$-orbits}
\label{sec:V0}

Passing, as done in the previous section, from a given compact open subgroup~$U$ to a subgroup $V$
satisfying \TA\ might reduce the index of displacement but may not minimise it. Indeed, the group $G$ in Example~\ref{ex:compact_example} is compact, so that the group itself is minimizing and the scale of $\alpha$ is~1, whereas the displacement index of $V$ is~$p$. The displacement index $\ind{\alpha(V)}{V\cap \alpha(V)}$ is larger than it might
be when an $\alpha$-orbit leaves $V$ and then returns to it, and in the present section it is shown that there is  a single compact, $\alpha$-stable subgroup, denoted $L_V$, that contains all such $\alpha$-orbits. In the next section it will be seen how to combine~$V$ with $L_V$ so as to preserve property \TA.

The subgroup $L_V$ is defined as follows. The first few definitions and results of this section hold for all compact, open subgroups of $G$, not just those satisfying \TA, and the groups are denoted by $U$ rather than $V$ to highlight that. 
\begin{definition} 
\label{de:curlyL}
Let $\alpha\in\End{G}$ and $U$ be a compact open subgroup of $G$. Define
$$
\mathcal{L}_U = \left\{ x \in G : \exists y\in U_+\text{ and } m,n\in
{\mathbb{N}}\text{ with } \alpha^{m}(y)=x 
 \text{ and }\alpha^n (x) \in U_-  \right\},
 $$
 and $L_U = \overline{\mathcal{L}_U}$.
\end{definition}
In words, an element $x$ belongs to $\mathcal{L}_U$ if and only if there is $y\in U_+$ such that $x$ is in the orbit $\{\alpha^n(y)\}_{n\geq0}$ and there is $N\geq0$ such that $\alpha^n(y)\in U$ for all $n\geq N$. A particular case that is important when $\alpha$ is an endomorphism is where $x\in G$ is equal to $\alpha^m(y)$ for some $y\in U_+$ and $x$ is in $\ker(\alpha^k)$ for some $k\geq0$. 

\begin{proposition}
\label{prop:Lsubgroup} The set ${\mathcal L}_U$ is an $\alpha$-stable subgroup of $G$, and $L_U$ is a subgroup such that $\alpha(L_U)\leq L_U$ and is dense.
\end{proposition}
\begin{proof}
It is clear that $\alpha({\mathcal L}_U) \subseteq {\mathcal L}_U$ and the reverse inclusion also holds because for each $y\in U_+$ there is $y'\in U_+$ such that $y = \alpha(y')$. That $\alpha(L_U)\leq L_U$ and is dense follows because $\alpha$ is continuous. Since $U_+$ and $U_-$ are groups,
${\mathcal L}_U$ is a subgroup of~$G$ and, consequently, so is $L_U$. 
\end{proof}
Every element of $\mathcal{L}_U$ has a bounded $\alpha$-regressive sequence by definition. It will follow from compactness of $L_U$ and the proposition that elements of $L_U$ also have bounded $\alpha$-regressive sequences, and hence in particular that $L_U$ is $\alpha$-stable. 

The proof  that ${L}_U$ is compact involves consideration of when the orbit $\{\alpha^n(y)\}_{n\geq0}$ leaves $U_+$ and the distance before it arrives in $U_-$. Here are the (temporarily) necessary definitions and notation.  

\begin{definition}
\label{de:exposure}
Let $U$ be a compact open subgroup of $G$.
\begin{enumerate}
\item Define, for $x\in {\mathcal L}_U$, 
$$
\lambda_U(x) = 
\begin{cases} 
-\max\{n\in\mathbb{N} \mid \alpha^n(x) \in U_+\}, & \text{ if } x\in U_+\\
\min\{n\in \mathbb{N}\mid \alpha^{-n}(x)\cap U_+\ne \emptyset \}, & \text{ if } x\not\in U_+
\end{cases}.
$$
\item The  \emph{exposure} of $x \in {\mathcal L}_U$, denoted $e_U(x)$, is the smallest non-negative integer $m$ such that there is $z\in U_+$ with: $x\in \{\alpha^j(z)\}_{j\in\mathbb{N}}$; $\alpha^k(z)\in U_+$; and 
$\alpha^{k+m+1} (z) \in U_-$.
\item ${\mathcal{E}}_U = 
\big(\alpha(U_+) \backslash U_+ \big)\cap {\mathcal L}_U$.
\end{enumerate}
\end{definition}

Every element of $\mathcal{L}_U$ belongs to an orbit that starts in $U_+$ and passes into $U_-$, where it stays, and the integer $\lambda_U(x)$ records the last point relative to $x$ where such an orbit belongs to $U_+$. Thus $\lambda_U(x)$ is equal to~1 if and only if $x$ is in ${\mathcal{E}}_U$. If $e_U(x) = 0$ there is an $\alpha$-orbit containing $x$ that never leaves $U$, in which case $x$ belongs to $U_+\cap U_-$. 

The set ${\mathcal E}_U$
consists of those points in ${\mathcal L}_U$ where $\alpha$-orbits first depart from $U$. For each $x\in {\mathcal L}_U\setminus U_+$ there is $z\in U_+$ with $\alpha(z)\in {\mathcal E}_U$ and $x = \alpha^{\lambda_U(x)}(z)$.  Note that, if $z\in \mathcal{E}_U$, then $\alpha^{e_U(z)}(z)\in U_-$ and $\alpha^k(z)\not\in U_-$ for every $k<e_U(z)$. 

\begin{lemma}
\label{lem:reduction}
 Let $x_1, x_2\in {\mathcal L}_U$, and suppose that there are  $y_i\in {\mathcal{E}}_U$ with $y_1y_2^{-1}$ belonging to $U_+$ and with $x_i = \alpha^m(y_i)$ for some $m\geq0$. Then $\nu_U(x_1x_2^{-1}) \leq m$.
\end{lemma}

\begin{proof}
By hypothesis, $x_1x_2^{-1} = \alpha^m(y_1y_2^{-1})$ where $y_1y_2^{-1}\in U_+$.
\end{proof}

The set $U_+{\mathcal{E}}_U$ consists of a finite number, $n_U$, of
$U_+$-cosets because ${\mathcal{E}}_U\subset \alpha(U_+)$ and
$\ind{\alpha(U_+)}{U_+}$ is finite. Choose from each of these
$U_+$-cosets a representative, $g_i$,
$i\in \{1, 2, \ldots, n_U\}$, such that $g_i$ belongs to ${\mathcal{E}}_U$
and satisfies
\begin{equation}
\label{eq:chooseg}
e_U(g_i) = \min \left \{ e_U(x): x \in {\mathcal{E}}_U \cap U_+ g_i
\right \}.
\end{equation}
Since $g_i\not\in U_+\cap U_-$, $e_U(g_i)\geq1$ for each $i$.

The compactness of $L_U$ will be an immediate consequence of the next result.

\begin{lemma}
\label{lem:factorscriptL}
Let $U$ be a compact open subgroup of $G$. Define 
$$
E_U=\max\left\{ e_U(g_i)\mid i\in\{1,2,\dots , n_U\}\right\}.
$$
Then 
$$
{\mathcal L}_U = (\alpha^{E_U}(U_+)
\cap {\mathcal L}_U)(U_-\cap {\mathcal L}_U).
$$
\end{lemma}
\begin{proof}
Let $x\in {\mathcal L}_U$. If $x\in \alpha^{E_U}(U_+)$ there is nothing to prove and so it may be assumed that $x = \alpha^{m_1}(z)$ for some $z\in {\mathcal{E}}_U$ with $m_1 = \nu_U(x)-1 \geq E_U$. Then $z\in U_+g_{i_1}$ for some $i_1\in \{1,2,\dots, n_U\}$ and $\nu_U(x\alpha^{m_1}(g_{i_1})^{-1}) \leq m_1$ by Lemma~\ref{lem:reduction}. Hence 
$$
x = x^{(1)}\alpha^{m_1}(g_{i_1})\text{ with }\nu_U(x^{(1)}) \leq m_1 < \nu_U(x)\text{ and }\alpha^{m_1}(g_{i_1})\in U_-.
$$
(That $\alpha^{m_1}(g_{i_1})\in U_-$ holds because $m_1 \geq E_U \geq e_U(g_{i_1})$.) If $x^{(1)}\in \alpha^{E_U}(U_+)$ the proof is complete. Otherwise, the argument may be repeated recursively to obtain $x^{(1)}, \dots, x^{(k)}$, with 
$$
 m_1 = \nu_U(x)-1 > m_2 = \nu_U(x^{(1)})-1  > \dots > m_k = \nu_U(x^{(k-1)})-1  > \nu_U(x^{(k)})-1,
$$ 
and $g_{i_1}, \dots , g_{i_k}$ such that 
$$
x = x^{(k)} \alpha^{m_k}(g_{i_k})\cdots \alpha^{m_1}(g_{i_1}).
$$
The recursion continues so long as $m_k \geq E_U$. For the first $k$ with $\nu_U(x_k)-1 < E_U$, we have $x^{(k)}\in \alpha^{E_U}(U_+)$ while $\alpha^{m_k}(g_{i_k})\cdots \alpha^{m_1}(g_{i_1})\in U_-$, so that $x$ belongs to $\alpha^{E_U}(U_+)U_-$ as claimed. 
\end{proof}

For $\alpha$ an automorphism, the next result would follow immediately from the lemma simply by applying~$\alpha^{E_U}$. A little more argument is required when $\alpha$ is not invertible. 
\begin{corollary}
\label{cor:factorscriptLagain}
Define $E_U$ as in Lemma~\ref{lem:factorscriptL}. Then 
$$
{\mathcal L}_U = (U_+
\cap {\mathcal L}_U)(\alpha^{-E_U}(U_-)\cap {\mathcal L}_U).
$$
\end{corollary}
\begin{proof}
Consider $x\in \mathcal{L}_U$. By Lemma~\ref{lem:factorscriptL}, $\alpha^{E_U}(x) = x_+x_-$ where $x_+$ belongs to $\alpha^{E_U}(U_+)
\cap {\mathcal L}_U$ and $x_-$ to $U_-\cap {\mathcal L}_U$. Since they belong to $\mathcal{L}_U$, there are $y_\pm\in U_+$ and $n> E_U$ such that $x_\pm = \alpha^n(y_\pm)$. Then 
$$
x = \alpha^{n-E_U}(y_+) \alpha^{n-E_U}(y_-)z,
$$ 
where $\alpha^{n-E_U}(y_+)\in U_+
\cap {\mathcal L}_U$, $\alpha^{n-E_U}(y_-)\in \alpha^{-E_U}(U_-)\cap {\mathcal L}_U$ and $z\in \ker(\alpha^{E_U})$. Since $z$ belongs to $\mathcal{L}_U$ and $\alpha^{E_U}(z) = \ident$, we have $z\in \alpha^{-E_U}(U_-)\cap {\mathcal L}_U$ as well, thus establishing the claim. 
\end{proof}

As foreshadowed, compactness of $L_U$ now follows. 
\begin{proposition}
\label{prop:Liscompact}
Let $U$ be a compact open subgroup of $G$.  Then $L_U$ is a compact
$\alpha$-stable subgroup of $G$ and 
$$
L_U = (\alpha^{E_U}(U_+)
\cap { L}_U)(U_-\cap { L}_U).
$$ 
\end{proposition}

\begin{proof}
Lemma~\ref{lem:factorscriptL} implies that $L_U$ is contained in the  compact set $\alpha^{E_U}(U_+)U_-$. Hence $L_U$ is compact. By Proposition~\ref{prop:Lsubgroup}, ${\mathcal L}_U$ is an $\alpha$-stable subgroup of $G$ and so ${\mathcal L}_U\leq \alpha(L_U)\leq L_U$. Since $L_U$ is compact, so is $\alpha(L_U)$ and it follows that $\alpha(L_U) = L_U$.

Since $(\alpha^{M_U}(U_+) \cap { L}_U)(U_-\cap { L}_U)$ lies between $\mathcal{L}_U$ and $L_U$, it is dense in $L_U$ and then, since $\alpha^{M_U}(U_+)
\cap { L}_U$ and $U_-\cap { L}_U$ are compact, their product is compact and equal to $L_U$.
\end{proof}

In the next section it will be important to know that $L_U$ does not change when passing from the compact, open subgroup $U$ to the open subgroup $V = U_{-n}$. That is seen in the next lemma.
\begin{lemma}
\label{lem:L_U=L_V}
Let $U$ be a compact, open subgroup of $G$ and suppose that $V = U_{-n}$ for some $n\geq0$. Then: 
\begin{enumerate}
\item $V_+ \leq U_+$ and $\bigcup_{k\in\mathbb{N}} \alpha^k(V_+) = \bigcup_{k\in\mathbb{N}} \alpha^k(U_+)$;
\label{lem:L_U=L_V1}
\item $V_- = U_-$; and
\label{lem:L_U=L_V2}
\item $\mathcal{L}_V = \mathcal{L}_U$ and $L_V = L_U$.
\label{lem:L_U=L_V3}
\end{enumerate}
\end{lemma}
\begin{proof}
(\ref{lem:L_U=L_V1}) It is clear that $V_+\leq U_+$ and that $\bigcup_{k\in\mathbb{N}} \alpha^k(V_+) \leq \bigcup_{k\in\mathbb{N}} \alpha^k(U_+)$. For the reverse inclusion, consider $\alpha^k(x)$ where $x\in U_+$. There is $y\in U_+$ such that $x = \alpha^n(y)$. Then $y\in U_{-n}$ and  $\alpha^k(x) = \alpha^{k+n}(y)$, which belongs to $\alpha^{k+n}(V_+)$. 

(\ref{lem:L_U=L_V2}) That $V_-\leq U_-$ is clear and that $V_-\geq U_-$ follows because $U_-\leq U_{-n} = V$. 

(\ref{lem:L_U=L_V3}) An element $x$ belongs to $\mathcal{L}_V$ if and only if $x\in \bigcup_{k\in\mathbb{N}}\alpha^k(V_+)$ and there is~$k$ in~$\mathbb{N}$ such that $\alpha^k(x)\in V_-$. That these criteria are equivalent to $x$ belonging to $\mathcal{L}_U$ follows from (\ref{lem:L_U=L_V1}) and (\ref{lem:L_U=L_V2}).
\end{proof}

It will be useful to have criteria for elements to belong to $L_U$. The next lemma is a first step in that direction.
\begin{lemma}
\label{lem:accpoint}
Let $U$ be a compact open subgroup of $G$ and let $x\in U_+$. Then every $\alpha$-regressive sequence $\{x_n\}_{n\in\mathbb{N}}$ for $x$ that is contained in $U_+$ has all of its accumulation points in $U_+\cap U_-$.
\end{lemma}
\begin{proof}
Since $x_n \in U_+\cap U_{(-n)}$ for each $n\in\mathbb{N}$, every accumulation point of $\{x_n\}_{n\in\mathbb{N}}$ belongs to $U_+\cap \bigcap_{n\geq0} U_{(-n)} = U_+\cap U_-$. 
\end{proof}

Lemma~\ref{lem:accpoint} is greatly extended in the next two results, which imply that any orbit that is bounded and enters a subgroup $V$ that is tidy above for $\alpha$ is entirely contained in $L_V$. Since the results have as a hypothesis that the compact, open subgroup is tidy above for $\alpha$, the subgroup will be denoted by $V$.    

\begin{lemma}
\label{lem:criterion}
Let $\alpha\in \End{G}$, let $V$ be a compact, open subgroup of $G$ that is tidy above for $\alpha$ and let $k\in \mathbb{N}$. Then $x\in \alpha^k(V_+)$ belongs to $L_V$ if and only if $\{\alpha^n(x)\}_{n\in\mathbb{N}}$ has an accumulation point. 
\end{lemma}
\begin{proof} If $x\in L_V$, then certainly $\{\alpha^n(x)\}_{n\in\mathbb{N}}$ has an accumulation point because $L_V$ is compact and $\alpha$-stable.

For the converse, let $x\in \alpha^k(V_+)$ and suppose that $\{\alpha^n(x)\}_{n\in\mathbb{N}}$ has an accumulation point, $c$ say.  Then $x = \alpha^k(x_1)$ where $x_1\in V_+$ and $\{\alpha^n(x_1)\}_{n\in\mathbb{N}}$ has $c$ as an accumulation point. Since $x$ belongs to $L_V$ if $x_1$ does, it may be assumed that $x\in V_+$. 
Let $N$ be a positive integer, and choose $m>N$ and $n > 2m$ such that $\alpha^m(x)$ and $\alpha^n(x)$ belong to $cV$. Then $\alpha^m(x)^{-1}\alpha^n(x)$ belongs to $V$ and there are $v_\pm\in V_\pm$ such that $\alpha^m(x)^{-1}\alpha^n(x) = v_+v_-$. There are $y_1\in V_+$ such that $\alpha^n(y_1) = v_+^{-1}$ and $y_2\in V_+$ such that $\alpha^{n-m}(y_2) = x^{-1}$. Put $x_N = y_1y_2x$. Then $x_N$ belongs to $V_+$ and $\alpha^n(x_N) = v_+^{-1}\alpha^m(x^{-1})\alpha^n(x) = v_-$, which belongs to~$V_-$. Hence $x_N\in \mathcal{L}_V$. Since the sequence $\{x_N\}_{N\geq0}$ is contained in $L_V$, which is compact by Proposition~\ref{prop:Liscompact}, it has an accumulation point, $l$ say, in $L_V$.

Considering $x_N$ for fixed $N$ once more, when $p<N$ we have 
\begin{equation*}
\label{eq:criterion}
\alpha^p(x_Nx^{-1}) = \alpha^p(y_1)\alpha^p(y_2) \in V_+.
\end{equation*}
Hence $\alpha^p(lx^{-1}) \in V_+$ for every $p\in\mathbb{N}$ and it follows that $lx^{-1}\in V_+\cap V_-$. Therefore $x$ lies in $L_V$ as claimed.
\end{proof}

The proof of the next result follows a similar pattern as that of the last but a separate proof is required because $\alpha$ is not invertible.  
\begin{lemma}
\label{lem:criterion2}
Let $\alpha\in \End{G}$ and let $V$ be a compact, open subgroup of $G$ that is tidy above for $\alpha$ and let $k\in\mathbb{N}$. Then $x\in \alpha^{-k}(V_-)$ belongs to $L_V$ if and only if $x$ has an $\alpha$-regressive sequence $\{x_n\}_{n\in\mathbb{N}}$ that has an accumulation point. 
\end{lemma}
\begin{proof}
Since $\alpha$ maps $L_V$ onto itself, each $x\in L_V$ has an $\alpha$-regressive sequence $\{x_n\}_{n\in\mathbb{N}}$ contained in $L_V$. This sequence has an accumulation point because $L_V$ is compact.

For the converse, let $\{x_n\}_{n\in\mathbb{N}}\in G$ be an $\alpha$-regressive sequence for $x$ and let $c$ be an accumulation point. Note that $\alpha^n(x_n) = x$ for each $n$. Let $N$ be a positive integer and choose $m>\max\{N,k\}$ and $n>2m$ such that $x_m,x_n\in cV$. Then $x_nx_m^{-1}\in V$ and there are $v_\pm\in V_\pm$ such that $x_nx_m^{-1} = v_+v_-$. Since $v_+\in V_+$ and $\alpha^{n+k}(x_nx_m^{-1}v_-^{-1})$, which equals $\alpha^k(x)\alpha^{n-m+k}(x^{-1})\alpha^{n+k}(v_-)$, is in $V_-$, the element $y_N := \alpha^n(v_+) = \alpha^{n}(x_nx_m^{-1}v_-^{-1})$ belongs to~$\mathcal{L}_V$. Moreover, $\alpha^j(x^{-1}\alpha^m(v_-))$ belongs to $V_-$ for every $j\geq k$ and $\alpha^{n-m}(x^{-1}\alpha^m(v_-)) = x^{-1}y_N$. Hence, putting 
\begin{equation}
\label{eq:z_N,p}
z_{N,p} = \alpha^{n-m-p}(x^{-1}\alpha^m(v_-))\text{ for each }p\in \{0,1,\dots, n-m-k\},
\end{equation}
we have, for each $p$ in  this set,
\begin{equation}
\label{eq:z_N,p2}
z_{N,p}\in V_-,\ \alpha(z_{N,p+1}) = z_{N,p}\text{ and }\alpha^p(z_{N,p}) = x^{-1}y_N.
\end{equation} 
Since $n-m > N$,~\eqref{eq:z_N,p} defines $z_{N,p}$ for at least all $N$ and $p$ with $N\geq p+k$. Define $z_{N,p} = \ident$ for those $N< p + k$ where it is not already defined by~\eqref{eq:z_N,p}. Then the sequence $\mathbf{z}_N : p\mapsto z_{N,p}$ belongs to $V_-^{\mathbb{N}}$ for each $N$. Hence $\{(\mathbf{z}_N,y_N)\}_{N\in \mathbb{N}}$ is a sequence in $V_-^{\mathbb{N}}\times L_V$. Since $L_V$ is compact, by Proposition~\ref{prop:Liscompact}, and $V_-$ is also compact, this sequence has an accumulation point, $(\mathbf{z},l)$ say,  where $\mathbf{z} = \{z_p\}_{p\in\mathbb{N}}$ belongs to $V_-^{\mathbb{N}}$ and $l$ to $L_V$. Then  it follows from~\eqref{eq:z_N,p2} that  $\mathbf{z}$ is a recursive sequence for $ x^{-1}l$ that is contained in $V_-$. In particular, $x^{-1}l$ belongs to $V_+\cap V_-$ and $x$, which equals $l(x^{-1}l)^{-1}$, belongs to $L_V$ as claimed.
\end{proof}  

The fact that a bounded $\alpha$-orbit that enters a tidy above subgroup $V$ is contained in $L_V$ may now be deduced.
\begin{proposition}
\label{prop:bounded_ in_LV}
Let $\alpha\in \End{G}$ and $V$ be tidy above for $\alpha$. Suppose that $\left\{ x_n\right\}_{n\in\mathbb{Z}}$ is an $\alpha$-orbit that enters $V$ and is such that $\left\{ x_n\right\}_{n\geq0}$ and $\left\{ x_n\right\}_{n\leq0}$ have accumulation points. Then $\left\{ x_n\right\}_{n\in\mathbb{Z}} \subset L_V$.
\end{proposition}
\begin{proof}
Suppose, by re-indexing the sequence, that $x_0\in V$. Then $x_0  = x_+x_-$ with $x_\pm\in V_\pm$. Since $x_+$ is in $V_+$, it has an $\alpha$-regressive sequence $\left\{y_n\right\}_{n\in\mathbb{N}}$ that is contained in $V_+$ and, since $x_-$ is in $V_-$, $\alpha^n(v_-)\in V_-$ for all $n\in\mathbb{N}$. Define, for $n\in\mathbb{Z}$, 
$$
x_n' = \begin{cases}
y_{-n}, & \text{ if }n\leq 0\\
x_n\alpha^n(x_-)^{-1}, & \text{ if } n>0
\end{cases}
\qquad\text{ and }\qquad
x_n'' =  \begin{cases}
y_{-n}^{-1}x_n, & \text{ if }n\leq 0\\
\alpha^n(x_-), & \text{ if } n>0
\end{cases}.
$$
Then $\left\{x_n'\right\}_{n\in\mathbb{Z}}$ is an $\alpha$-orbit with $x_n'\in V_+$ for all $n\leq0$ and $\left\{ x_n'\right\}_{n\geq0}$ having an accumulation point, and $\left\{x_n''\right\}_{n\in\mathbb{Z}}$ is an $\alpha$-orbit $x_n''\in V_-$ for all $n\geq0$ and $\left\{ x_n''\right\}_{n\leq0}$ having an accumulation point. Hence $\left\{x_n'\right\}_{n\in\mathbb{Z}}\subset L_V$ and $\left\{x_n''\right\}_{n\in\mathbb{Z}}\subset L_V$ by Lemmas~\ref{lem:criterion} and~\ref{lem:criterion2}. Since $x_n = x_n'x_n''$ for all $n\in\mathbb{Z}$, it follows that $\left\{x_n\right\}_{n\in\mathbb{Z}}$ is contained in~$L_V$. (Another version of this factoring argument is used below in Lemma~\ref{lem:cut}.)
\end{proof} 
Since $\ident$ belongs to every subgroup, the following is a special case of Lemma~\ref{lem:criterion2} and Proposition~\ref{prop:bounded_ in_LV}.
\begin{corollary}
\label{cor:bounded_ in_LV}
Let $x\in\ker(\alpha^k)$ for some $k\in \mathbb{N}$ and suppose that $x$ has an $\alpha$-regressive sequence that has an accumulation point. Then $x\in L_V$ for every compact, open subgroup $V$ that is tidy above for $\alpha$. 
\end{corollary}

\section{A subgroup that is tidy above and contains bounded $\alpha$-orbits}
\label{sec:tidy_below}

For a given subgroup $V$ that is tidy above for $\alpha$, the compact subgroup $L_V$ by definition contains all $\alpha$-orbits that move from $V_+$ to $V_-$, and it has been proved that it in fact contains all bounded $\alpha$-orbits that intersect $V$. The aim in this section is to combine $V$ and $L_V$ so as to produce another compact open subgroup, $W$, that is tidy above and contains $L_W$. This is done in such a way that 
$$
[\alpha(W) : \alpha(W)\cap W] \leq [\alpha(V) : \alpha(V)\cap V]
$$ 
with equality if and only if $V$ already contains $L_V$.  

Care is required in order to combine $V$ and $L_V$. The most obvious way to do so is to form $\langle V,L_V\rangle$, but that group need not be compact. The next most obvious is that, since $L_V$ is compact, $V$ has an open subgroup, $\widehat{V}$ say, that is normalized by $L_V$ and $\widehat{V}L_V$ is then compact, open and contains $L_V$. However, as discussed in \cite{wi:further}, the displacement index of $\widehat{V}$ may be greater than that of $V$ and the proof that the displacement index of $\widehat{V}L_V$ is then less than that of $V$ is complicated. The approach of the next lemma defines an open subgroup of $V$ that contains~$\widehat{V}$.

\begin{lemma}
\label{lem:commute_groups}
Let $K$ and $L$ be closed subgroups of the topological group $G$ and suppose that $K$ is  compact and $K\cap L$ has finite index in $L$. Put
$$
\widetilde{K} = \left\{ x\in K \mid xL\subseteq LK\right\}.
$$
Then $\widetilde{K}$ is an open, finite index subgroup of $K$ and $\widetilde{K}L = L\widetilde{K}$. 
\end{lemma}
\begin{proof}
That $\widetilde{K}$ is a closed sub-semigroup of $K$ is easily verified. Then $\widetilde{K}$ is a closed subgroup of $K$ because compact sub-semigroups of a group are stable under the inverse. That $\widetilde{K}$ is open follows because it contains the open subgroup $\bigcap_{l\in L}lKl^{-1}$.  

To see that $\widetilde{K}L = L\widetilde{K}$, consider $x\in \widetilde{K}$ and $l\in L$. By definition of $\widetilde{K}$, $xl = l'x'$ for some $l'\in L$ and $x'\in K$. Then $x'\in \widetilde{K}$ because, for $l_1\in L$, we have 
$$
x'l_1 = (l')^{-1}x(ll_1) = (l')^{-1}l_2 x''\in LK
$$
since $x\in \widetilde{K}$. Hence $x'\in \widetilde{K}$ as claimed and it has been shown that $\widetilde{K}L \subseteq L\widetilde{K}$. The reverse inclusion may be shown to hold by applying the inverse map and using that $\widetilde{K}$ and $L$ are groups. 
\end{proof}

For the remainder of this section: $V$ is a compact open subgroup that is tidy above for $\alpha$; and $L_V$ is as in Definition~\ref{de:curlyL}. Then 
\begin{equation}
\label{eq:defnV'}
\widetilde{V} := \left\{ x\in V\mid xL_V\subseteq L_VV\right\}
\end{equation}
is an open subgroup of $V$ by Lemma~\ref{lem:commute_groups}. The next few results establish properties of $\widetilde{V}$ that lead to the calculation of its displacement index under $\alpha$. 

\begin{lemma}
\label{lem:V+-minV'}
The subgroup $V_+\cap V_-$ is contained in $\widetilde{V}$ and there is $m>0$ such that $V_+\cap V_{-m}\leq \widetilde{V}$.
\end{lemma}
\begin{proof}
The subgroup $V_+\cap V_-$ is contained in $V\cap L_V$ and therefore in $\widetilde{V}$. That there is an $m>0$ such that $V_+\cap V_{-m}\leq \widetilde{V}$ then follows because $\{V_+\cap V_{-m}\}_{m\in \mathbb{N}}$ is a decreasing sequence of compact sets whose intersection is $V_+\cap V_-$ and $\widetilde{V}$ is an open neighbourhood of $V_+\cap V_-$.
\end{proof}

\begin{lemma}
\label{lem:tildeVcapV+}
The group $\widetilde{V}$ satisfies
\begin{enumerate}
\item \label{eq:tildeVcapV+1}
$\widetilde{V}_+ = \widetilde{V}\cap V_+ = \left\{ v\in V_+ \mid vL_V\subseteq L_VV_+\right\}$\ \ and 
\item
\label{eq:tildeVcapV+2}
$\widetilde{V}_- = \widetilde{V}\cap V_- = \left\{ v\in V_- \mid vL_V\subseteq L_VV_-\right\}$. 
\end{enumerate}
\end{lemma} 
\begin{proof}
Although the statements of (\ref{eq:tildeVcapV+1}) and (\ref{eq:tildeVcapV+2}) are similar and their proofs follow the same pattern, the proofs do differ and will be given separately.

We begin with~(\ref{eq:tildeVcapV+1}). It is clear that $\widetilde{V}_+ \leq \widetilde{V}\cap V_+$. The proof will be completed by showing that 
$$
\widetilde{V}\cap V_+ \leq \left\{ v\in V_+ \mid vL_V\subseteq L_VV_+\right\} \leq \widetilde{V}_+.
$$

For the first inclusion, let $v\in \widetilde{V}\cap V_+$ and, towards showing that $vL_V\subseteq L_VV_+$, let $l\in L_V$. It may be supposed that in fact $l\in\mathcal{L}_V$ because  $\mathcal{L}_V$ is dense in $L_V$ and $L_VV_+$ is closed. Since $v\in \widetilde{V}$, $vl = l_1v_1$ with $l_1\in L_V$ and $v_1\in V$ and, since $V$ satisfies \TA($\alpha$), $v_1 = v_-v_+$ with $v_\pm \in V_\pm$. Then $vlv_+^{-1} = l_1v_-$, where $vlv_+^{-1}$ belongs to $\alpha^k(V_+)$ for some $k\in\mathbb{Z}$ because $l\in\mathcal{L}_V$. Moreover, $\left\{\alpha^m(l_1v_-)\right\}_{m\in\mathcal{N}}$ is bounded. Hence, by Lemma~\ref{lem:criterion}, $l_1v_-\in L_V$ and it has been shown that $vl = (l_1v_-)v_+$ which belongs to $L_VV_+$. 

For the second inclusion, suppose that $v\in V_+$ satisfies $vL_V\subseteq L_VV_+$.  Then $v$ belongs to $\widetilde{V}\cap V_+$. To complete the proof it suffices to show that there is $y\in \widetilde{V}\cap V_+$ such that $v = \alpha(y)$, for an induction argument will then show that there is an $\alpha$-regressive sequence $\{y_n\}_{n\in\mathbb{N}}\subset \widetilde{V}$ for $v$. There is certainly $y\in V_+$ with $\alpha(y) = v$. It will be shown that $y\in \widetilde{V}$. To this end, let $l\in L_V$, where once again it may be supposed that $l$ is in fact in $\mathcal{L}_V$, and consider~$yl$. Then $\alpha(y)\alpha(l) = v\alpha(l) = l'u$ with $l'\in L_V$ and $u\in V_+$. There are $l_1\in L_V$ with $\alpha(l_1) = l'$ and $u_1\in V_+$ with $\alpha(u_1) = u$, and so $yl = zl_1u_1$ for some $z\in \ker(\alpha)$. Hence $ylu_1^{-1} = zl_1$, where $ylu_1^{-1}\in \alpha^k(V_+)$ for some $k$ and $\{\alpha^n(zl_1)\}_{n\in\mathbb{N}}$ is bounded. Hence, by Lemma~\ref{lem:criterion}, $zl_1\in L_V$ and $yl \in L_VV_+$, which implies that $y\in \widetilde{V}$ as claimed. 

The proof of~(\ref{eq:tildeVcapV+2}) follows the same pattern but is easier. Again it is clear that $\widetilde{V}_- \leq \widetilde{V}\cap V_-$ and the proof may be completed by showing that 
$$
\widetilde{V}\cap V_- \leq \left\{ v\in V_- \mid vL_V\subseteq L_VV_-\right\} \leq \widetilde{V}_-.
$$

Consider $v\in \widetilde{V}\cap V_-$ and $l\in L_V$. Then $vl = l_1w$ with $l_1\in L_V$ and $w\in V$ and it must be shown that $w$ may be chosen in $V_-$. Since $V$ satisfies \TA($\alpha$), $w = w_+w_-$ with $w_\pm \in V_\pm$. Then 
$$
w_+ = l_1^{-1}vlw_-^{-1} \in L_VV_-L_VV_-,
$$
and it follows that $\{\alpha^n(w_+)\}_{n\in\mathbb{N}}$ is bounded, whence $w_+\in L_V$ by Lemma~\ref{lem:criterion}. Hence $vl = (l_1w_+)w_-\in L_V V_-$. 

Next, consider $v\in V_-$ such that $vL_V \subseteq L_VV_-$. For each $n\in \mathbb{N}$ we have 
$$
\alpha^n(v)L_V = \alpha^n(vL_V) \subseteq \alpha^n(L_V V_-) = L_V \alpha^n(V_-) \subseteq L_V V.
$$
Hence $\alpha^n(v) \in \widetilde{V}$ for each $n\in\mathbb{N}$, that is, $v$ belongs to $\widetilde{V}_-$. 
\end{proof}

\begin{lemma}
\label{lem:V'TA}
The group $\widetilde{V}$ is tidy above for $\alpha$ and 
$$
[\alpha(\widetilde{V}) : \alpha(\widetilde{V})\cap \widetilde{V}] = [\alpha(V) : \alpha(V) \cap V].
$$
\end{lemma}
\begin{proof}
Let $v\in \widetilde{V}$. Then $v = v_-v_+$ with $v_\pm \in V_\pm$. It must be shown that $v_\pm$ belong to $\widetilde{V}_\pm$ respectively. It suffices to show that $v_+\in \widetilde{V}_+$ because then $v_-$ belongs to $\widetilde{V}\cap V_-$, which equals $\widetilde{V}_-$ by Lemma~\ref{lem:tildeVcapV+}. 

Let $l\in L_V$. With the aid of Lemma~\ref{lem:V+-minV'}, choose $m>0$ such that $V_+\cap V_{-m}\leq \widetilde{V}$, and then choose $x\in V_+\cap V_{-m}$ such that $v_+ = \alpha^m(x)$. Since $\alpha$ is an onto map on $L_V$, there is also $l'\in L_V$ such that $l = \alpha^m(l')$. Then $x\in \widetilde{V}_+$ by Lemma~\ref{lem:tildeVcapV+}, and so $xl' = l_2x_2$ for some $l_2\in L_V$ and $x_2\in V_+$. Hence
\begin{equation}
\label{eq:firstv+}
v_+l = \alpha^m(l_2)\alpha^m(x_2).
\end{equation}
Although $\alpha^m(l_2)\in L_V$, it is not yet shown that  $\alpha^m(x_2)$ belongs to $V_+$ and so~(\ref{eq:firstv+}) does not suffice by itself to show that $v_+\in \widetilde{V}_+$. For this extra step, recall that $v\in\widetilde{V}$, so that $v_-v_+l = l_1u_-u_+$ for some $l_1\in L_V$ and $u_\pm\in V_\pm$. Hence
\begin{equation}
\label{eq:secondv+}
v_+l = v_-^{-1}l_1u_-u_+.
\end{equation}
Equations~(\ref{eq:firstv+}) and (\ref{eq:secondv+}) imply that $\alpha^m(l_2)\alpha^m(x_2) = v_-^{-1}l_1u_-u_+$, whence 
 $$
 \alpha^m(x_2)u_+^{-1} = \alpha^m(l_2)^{-1}v_-^{-1}l_1u_- \in \alpha^m(V_+)\cap L_VV_-L_VV_-.
 $$
 Hence $\{\alpha^n(\alpha^m(x_2)u_+^{-1})\}_{n\in\mathbb{N}}$ is bounded and Lemma~\ref{lem:criterion} implies that $\alpha^m(x_2)u_+^{-1}$ belongs to $L_V$. Therefore $\alpha^m(x_2)\in L_VV_+$, whence it follows from~(\ref{eq:firstv+}) that $v_+\in\widetilde{V}$ and then from Lemma~\ref{lem:tildeVcapV+} that $v_+$ belongs to $\widetilde{V}_+$. 
 
Since $V$ is tidy above for $\alpha$, Lemma~\ref{lem:stable} shows that
\begin{equation}
\label{eq:stable_equality}
[\alpha(V) : \alpha(V)\cap V] = [V_+ \cap V_{-m}: V_+\cap V_{-m-1}]
\end{equation}
for every $m\geq0$, and similarly for $\widetilde{V}$ by what has just been shown. By Lemma~\ref{lem:V+-minV'}, there is $m\geq 0$ such that $V_+\cap V_{-m} \leq \widetilde{V}$. Then it follows from Lemma~\ref{lem:tildeVcapV+} that ${V}_+\cap {V}_{-m} \leq \widetilde{V}_+$ and we have
\begin{align*}
V_+\cap V_{-m-1} &\leq {V}_+\cap {V}_{-m} \leq \widetilde{V}_+\\
\text{ and }\quad 
V_+\cap V_{-m-1} &\leq \widetilde{V}_+\cap \widetilde{V}_{-1} \leq \widetilde{V}_+.
\end{align*}
Hence 
\begin{eqnarray*}
&\ [\widetilde{V}_+ : {V}_+\cap {V}_{-m}][{V}_+\cap {V}_{-m} : V_+\cap V_{-m-1}]\\
\text{ and }\quad 
&[\widetilde{V}_+ : \widetilde{V}_+\cap \widetilde{V}_{-1}] [\widetilde{V}_+\cap \widetilde{V}_{-1} : V_+\cap V_{-m-1}]
\end{eqnarray*} are both equal to $[\widetilde{V}_+ : V_+\cap V_{-m-1}]$ 
and, by~(\ref{eq:stable_equality}),
$$
[\widetilde{V}_+ : {V}_+\cap {V}_{-m}][\alpha(V) : \alpha(V)\cap V]  = 
[\alpha(\widetilde{V}) : \alpha(\widetilde{V}) \cap \widetilde{V}] [\widetilde{V}_+\cap \widetilde{V}_{-1} : V_+\cap V_{-m-1}].
$$
The claimed equation will follow once it has been shown that
$$
[\widetilde{V}_+ : {V}_+\cap {V}_{-m}]  = 
[\widetilde{V}_+\cap \widetilde{V}_{-1} : V_+\cap V_{-m-1}].
$$
For this, observe that the map $v(V_+\cap V_{-m-1}) \mapsto \alpha(v)({V}_+\cap {V}_{-m})$ is well-defined and onto from $(\widetilde{V}_+\cap \widetilde{V}_{-1})/(V_+\cap V_{-m-1})$ to $\widetilde{V}_+/({V}_+\cap {V}_{-m})$ and that it is one-to-one because, if $v\in \widetilde{V}_+$ satisfies $\alpha(v) \in V_+\cap V_{-m}$, then $v$ belongs to $V_+\cap V_{-m-1}$. 
\end{proof} 

\begin{proposition}
\label{prop:define_W} 
Put $W = \widetilde{V}L_V$. Then $W$ is a compact, open subgroup of $G$ and $\widetilde{V}_+L_V$ and $\widetilde{V}_-L_V$ are compact subgroups of $W$. Furthermore,
\begin{enumerate}
\item $W = W_+W_-$, that is, $W$ is tidy above for $\alpha$; 
\label{prop:define_W1} 
\item $W_\pm = \widetilde{V}_\pm L_V$; 
\label{prop:define_W2} 
\item $L_{\widetilde{V}} = L_V = W_+\cap W_-  = \mathcal{L}_W = L_W$; and
\label{prop:define_W3} 
\item $[\alpha(W) : \alpha(W)\cap W] \leq [\alpha(V) : \alpha(V)\cap V]$ with equality if and only if $L_V\leq V$ (and $W = V$).
 \label{prop:define_W4} 
\end{enumerate}
\end{proposition}
\begin{proof}
That $W$ is a subgroup of $G$ follows from Lemma~\ref{lem:commute_groups}. It is an open subgroup because $\widetilde{V}$ is open, and is compact because $\widetilde{V}$ and $L_V$ are. That $\widetilde{V}_+L_V$ and $\widetilde{V}_-L_V$ are subgroups follows from Lemma~\ref{lem:tildeVcapV+}.

(\ref{prop:define_W1}) It is immediate from the definitions that $\alpha(\widetilde{V}_+L_V) \geq \widetilde{V}_+L_V$ and $\alpha(\widetilde{V}_-L_V) \leq \widetilde{V}_-L_V$. Hence $\widetilde{V}_+L_V\leq W_+$ and $\widetilde{V}_-L_V \leq W_-$. Then
$$
W = \widetilde{V}L_V = \widetilde{V}_+\widetilde{V}_-L_V = \left(\widetilde{V}_+L_V\right)\left(\widetilde{V}_-L_V\right) \leq W_+W_-.
$$ 

(\ref{prop:define_W2}) It has already been seen that $W_\pm \geq \widetilde{V}_\pm L_V$. 

Let $w \in W_+$. Then $w = vl$ for $v\in \widetilde{V}$ and $l\in L_V$ and, since $\widetilde{V}$ is tidy above, $v = v_+v_-$ with $v_\pm\in \widetilde{V}_\pm$. Hence $v_- =  v_+^{-1}wl^{-1}$, where $v_+^{-1}wl^{-1}$ belongs to $W_+$, and hence $v_-$ is in $L_W$. It follows that there is an $\alpha$-regressive sequence $\{x_n\}_{n\in\mathbb{N}}$ for~$v_-$ in the compact set $L_W$ whence, by Lemma~\ref{lem:criterion2},  $v_-$ belongs to~$L_V$. Therefore $w = v_+(v_-l)$ and belongs to $\widetilde{V}_+L_V$. 

Next, let $w\in W_-$. Then $w = v_-v_+l$ with $v_\pm\in \widetilde{V}_\pm$ and $l\in L_V$. Hence $v_+ = v_-^{-1}wl^{-1}\in W_+\cap W_-$ and belongs to $L_W$. Since $L_W$ is compact, it follows from Lemma~\ref{lem:criterion} that $v_+$ is in $L_V$, whence $w = v_-(v_+l)$ belongs to $\widetilde{V}_-L_V$. 

(\ref{prop:define_W3}) It is clear that
$$
L_{\widetilde{V}} \leq L_V \leq W_+\cap W_-  \leq \mathcal{L}_W \leq L_W.
$$
It is also easy to see that $L_{\widetilde{V}} \geq L_V$: if $v\in \mathcal{L}_V$, then $v = \alpha^n(x)$ for $x\in V_+\cap V_{-m}$, so that $x\in \widetilde{V}_+$ by Lemmas~\ref{lem:V+-minV'} and~\ref{lem:tildeVcapV+}, whence, since $\{\alpha^n(x)\}_{\mathbb{N}} \subset L_V$ is bounded, $v\in \mathcal{L}_{\widetilde{V}}$ by Lemma~\ref{lem:criterion}. Since $L_{{V}}$ is closed, it suffices, in order to complete the proof, to show that $\mathcal{L}_W \leq L_{{V}}$. For this consider $w\in \mathcal{L}_W$. Then $w = \alpha^n(w_+)$ for some $n\in\mathbb{N}$ and $w_+$ in $W_+$. By~(\ref{prop:define_W2}), $w_+ = v_+l$ with $v_+\in \widetilde{V}_+$ and $l\in L_V$ and so it suffices to show that $v_+\in \widetilde{V}_+\cap \mathcal{L}_W$ belongs to $L_{{V}}$. That follows from Lemma~\ref{lem:criterion} however, because $\mathcal{L}_W$ has compact closure, by Lemma~\ref{prop:Liscompact}.

(\ref{prop:define_W4})  Since $W$ and $\widetilde{V}$ are tidy above, Lemmas~\ref{lem:stable} and~\ref{lem:V'TA} show that the claimed inequality is equivalent to
$$
[W_+ : W_+\cap W_{-1}] \leq [\widetilde{V}_+ : \widetilde{V}_+\cap \widetilde{V}_{-1}].
$$
To see the latter, consider the map $\phi : \widetilde{V}_+/ (\widetilde{V}_+\cap \widetilde{V}_{-1}) \to W_+ / (W_+\cap W_{-1})$ given by 
\begin{equation}
\label{eq:phi}
\phi : v(\widetilde{V}_+\cap \widetilde{V}_{-1}) \mapsto v(W_+\cap W_{-1}).
\end{equation} 
This map $\phi$ is well-defined because $\widetilde{V}_+\leq W_+$ and  $\widetilde{V}_+\cap \widetilde{V}_{-1} \leq W_+\cap W_{-1}$: the inequality holds because $\phi$ is onto,  which follows from (\ref{prop:define_W2}) and because $L_V$ is contained in $W_+\cap W_{-1}$. 

If $L_V \leq V$, then $W = V$ and equality holds. If $L_V \not\leq V$, then $L_{\widetilde V} \not\leq \widetilde{V}$ and there is $v_+\in \widetilde{V}_+\setminus \widetilde{V}_{-1}$ such that $v_+\in L_{\widetilde{V}} \leq W_+\cap W_{-1}$. Then the map $\phi$ in (\ref{eq:phi}) is not one-to-one and $[\alpha(W) : \alpha(W)\cap W]$ is strictly less than $[\alpha(V) : \alpha(V)\cap V]$. 
\end{proof}

\section{Criteria for a subgroup to contain bounded $\alpha$-orbits}
\label{sec:TB_criteria}

The condition that $\bigcup_{n\geq0} \alpha^n(W_+)$ be closed, called `Property {\bf T2}' in~\cite{Wi:structure}, is necessary for the subgroup $W$ to be minimizing for the automorphism $\alpha$, and is sufficient in conjunction tidiness above. This condition does not guarantee that $W$ contains bounded $\alpha$-orbits or suffice to ensure that $W$ is minimizing when $\alpha$ is an endomorphism however, as the following example shows. In this section, an additional criterion is identified that, together with tidiness above and closedness of $\bigcup_{n\geq0} \alpha^n(W_+)$, does ensure that $W$ contains bounded $\alpha$-orbits and, it will be seen in the next section, is minimising for $\alpha$. 
\begin{example}
\label{ex:not_TB}
Let $G$ be the additive group of the compact ring $F_p[[t]]$ and identify this group with $F_p^{\mathbb{N}}$. For $S\subseteq \mathbb{N}$, identify $F_p^S$ with the subgroup of $F_p^{\mathbb{N}}$ consisting of sequences with support in $S$.  Define $\alpha\in\End{G}$ by
$$
\alpha(g)_n = g_{n+1}, \quad (g\in F_2^{\mathbb{N}}).
$$
Then $\alpha$ is onto but is not an automorphism because $\ker(\alpha) = F_2^{\{0\}}$. Since $G$ itself is compact and $\alpha$-stable, $s(\alpha)=1$. 

Put $V = F_2^{\mathbb{N}\setminus \{0,2\}}$. Then $V$ is an open subgroup of $G$ and $[\alpha(V) : \alpha(V)\cap V] = 4$. An easy calculation shows that $V_+ = F_2^{\mathbb{N}\setminus \{0,1,2\}}$, which is an open subgroup of $G$, and that $V_- = \triv$. Hence $V_+$ satisfies \TA\ and $[\alpha(V_+) : V_+] = 2$, so that $V_+$ is not minimizing. However $V_+$ satisfies the condition {\bf T2} of~\cite{Wi:structure} because $\bigcup_{n\geq0} \alpha^n(V_+)$ is equal to $G$, which is closed. 

In this example $\mathcal{L}_V$ comprises those elements of $g\in G$ such that $\alpha^n(g) = \ident$ for some $n\geq0$, that is, the subgroup of $G$ consisting of sequences with finite support. This subgroup is dense in $G$. Hence $L_V = G$. 
\end{example}

As for automorphisms, the criteria desired for endomorphisms relate to the dilation of $V_+$ by $\alpha$. 
\begin{definition}
\label{defn:Vplusplus}
For each compact open subgroup, $V$, of $G$ define
$$
V_{++} = \bigcup_{n\geq0} \alpha^n(V_+).
$$
\end{definition} 
\noindent Note that $V_{++}$ is a
subgroup of $G$ because the fact that $\alpha^{n+1}(V_+)\geq \alpha^n(V_+)$ for each $n$ implies that it is an increasing union of subgroups. 

The additional criterion needed for endomorphisms relates to a subgroup that is contained in  $L_V$ and is trivial when $\alpha$ is an automorphism. 
\begin{definition}
\label{defn:KV}
For each compact open subgroup, $V$, of $G$ define
$$
\mathcal{K}_V = \left\{ g\in G \mid \exists v\in V_+ \text{ and }m,n\in\mathbb{N} \text{ such that } g = \alpha^m(v) \text{ and }\alpha^n(v) = \ident\right\}
$$
	and $K_V = \overline{\mathcal{K}_V}$. 
\end{definition}
It is clear that $K_V$ is a subgroup and that it is not contained in $V$ in Example~\ref{ex:not_TB}. Here is a criterion for $K_V$ to be contained in $V$. 
\begin{proposition}
\label{prop:TB_criterion1}
Let $V$ be a compact, open subgroup of $G$ that is tidy above for the endomorphism $\alpha$.
Then $K_V$ is contained in $L_V$ and is compact, $\alpha$-stable and normal in the larger group $\overline{V_{++}}$. 

The subgroup $K_V$ is contained in $V$ if and only if the sequence $$
\left\{ [\alpha^{n+1}(V_+) : \alpha^n(V_+)]\right\}_{n\in\mathbb{N}}
$$ 
is constant. 
\end{proposition}
\begin{proof}
That $\mathcal{K}_V\leq \mathcal{L}_V$ is immediate from the definitions because $\triv\leq V_-$.  Since $L_V$  is compact, by Proposition~\ref{prop:Liscompact}, it follows that $K_V$ is as well. 

It is clear that $\alpha(\mathcal{K}_V)\leq \mathcal{K}_V$ and hence that $K_V$ is invariant under $\alpha$. To see that in fact $\alpha(\mathcal{K}_V) = \mathcal{K}_V$, consider $g\in \mathcal{K}_V$. By definition, $g = \alpha^m(v)$ for some $v\in V_+$ and $\alpha^n(v) = \ident$. It may be assumed that $m\geq1$ because $\alpha$ maps $V_+$ onto itself. Then $g_1 = \alpha^{m-1}(v)$ belongs to $\mathcal{K}_V$ and $\alpha(g_1) = g$ and it has been shown that $\mathcal{K}_V$ is stable under $\alpha$. It follows that $\alpha(K_V)$ is dense in~$K_V$ and then, since $K_V$ is compact, that $\alpha(K_V)=K_V$. 

To see that $K_V$ is normalised by $\overline{V_{++}}$, let $g\in \mathcal{K}_V$ and $x\in V_{++}$, so that $g = \alpha^m(v)$ for some $v\in V_+$ with $\alpha^n(v) = \ident$ and $x = \alpha^r(y)$ for some $y\in V_{+}$. By increasing either $m$ or $r$ if necessary and replacing either $v$ or $y$ with another element of $V_+$, it may be assumed that $m=r$. Then $xgx^{-1} = \alpha^m(yvy^{-1})$, where $yvy^{-1}\in V_+$, and $\alpha^{n-}(yv_1y^{-1}) = \ident$. Hence $\mathcal{K}_V$  is normal in $V_{++}$ and it follows that $K_V$ is normal in $\overline{V_{++}}$. 

The endomorphism $\alpha$ induces, for each $n\in\mathbb{N}$,  a well-defined map 
$$
[\alpha]_n : \alpha^{n+1}(V_+)/ \alpha^n(V_+) \to \alpha^{n+2}(V_+)/ \alpha^{n+1}(V_+).
$$ 
Since $[\alpha]_n$ clearly maps onto, $\left\{ [\alpha^{n+1}(V_+) : \alpha^n(V_+)]\right\}_{n\in\mathbb{N}}$ is a non-increasing sequence of positive integers. Moreover, $[\alpha^{n+2}(V_+) : \alpha^{n+1}(V_+)]$ is strictly less than $[\alpha^{n+1}(V_+) : \alpha^n(V_+)]$ if and only if $[\alpha]_n$ is not one-to-one. The latter is equivalent to the existence of $v_1\alpha^n(V_+) \ne v_2\alpha^n(V_+)$ in $\alpha^{n+1}(V_+)/\alpha^n(V_+)$ such that  $\alpha(v_1)\alpha^{n+1}(V_+) = \alpha(v_2)\alpha^{n+1}(V_+)$. Since $\alpha^{n+1}(V)$ contains $\alpha^n(V)$, it may be supposed that $\alpha(v_1) = \alpha(v_2)$. Then $v_1v_2^{-1}$ belongs to $\ker(\alpha)\cap \alpha^{n+1}(V_+)\setminus \alpha^n(V_+)$, whence $v_1v_2^{-1}\in K_V\setminus V$. 
\end{proof}
Since $K_V$ is trivial if $\alpha$ is one-to-one, the following is immediate.
\begin{corollary}
\label{cor:TB_criterion1}
If the endomorphism $\alpha$ is one-to-one and $V$ is tidy above for $\alpha$, then 
$$
\left\{ [\alpha^{n+1}(V_+) : \alpha^n(V_+)]\right\}_{n\in\mathbb{N}}
$$ 
is constant.
\endproof
\end{corollary}
For endomorphisms that are not one-to-one, the following consequence of  Corollary~\ref{cor:bounded_ in_LV} produces elements of $K_V$. 
\begin{corollary}
\label{cor:in_KV}
Let $x\in \ker(\alpha^k)$ for some $k\in\mathbb{N}$ and suppose that $x$ has an $\alpha$-regressive sequence that has an accumulation point. Then $x\in K_V$ for any compact, open subgroup $V$ that is tidy above for $\alpha$ and contains $L_V$.
\end{corollary}
\begin{proof}
It follows from Corollary~\ref{cor:bounded_ in_LV} that $x\in L_V$. The hypothesis that $L_V$ is contained in $V$ then implies that $x$ is in $V_+$, whence $x$ belongs to~$K_V$.
\end{proof} 

The following technical lemma is an extension to endomorphisms of the corresponding result for automorphisms proved in~\cite{Wi:structure}. It is used in the proof that it is necessary for $V_{++}$ to be closed in order for bounded $\alpha$-orbits to be contained in $V$. 
\begin{lemma}
\label{lem:cut} 
Let $V$ be a compact, open subgroup of $G$ that is tidy above for
$\alpha$. Suppose that $w \in G$ is such that  $\alpha^m (w)$ and
$\alpha^n (w)$ belong to $V$, where $0\leq m \leq n$. 
 Then $w = yz$, where
\begin{eqnarray*}
\label{eq:cut1}
&&\alpha^m (y) \in V_+   \text{ and }\alpha^n (y) \in V_-
\\
\label{eq:cut2}
 {\text{ and }}
&&\alpha^k (z) \in V   \text{ for }\ m \leq k
\leq n.
\end{eqnarray*}
\end{lemma}

\begin{proof}  Begin by using tidiness above to factor $\alpha^m (w)$ as
\begin{equation*}
\label{eq:firstfactor}
\alpha^m (w) = uv,\text{ with } u \in V_+ \text{ and } v \in V_- .
\end{equation*}
Then, since $u\in V_+$, there is $u_0\in V_+$ such that $u = \alpha^m(u_0)$ and  
$$
\alpha^n (w) = \alpha^{n}(u_0)\alpha^{n-m}(v).
$$
Observe that $\alpha^{n}(u_0)$
belongs to $V$ because $\alpha^n (w)$ is in $V$, by hypothesis, and $\alpha^{n-m}(v)$ is in $V_-$. Hence,
again by property \TA($\alpha$),
$\alpha^n (u_0)$ can be factored as 
\begin{equation*}
\label{eq:secondfactor}
\alpha^n (u_0) = st, \text{ where } s \in V_-   \text{ and } t \in V_+.
\end{equation*}
Since $t\in V_+$, there is $t_0\in V_+$ such that $\alpha^n(t_0) = t$. 

Put $y = u_0t_0^{-1}$. Then 
\begin{align*}
\alpha^m(y) &= \alpha^m(u_0)\alpha^m(t_0^{-1}) = u\alpha^m(t_0^{-1}),\\ 
\intertext{which belongs to $V_+$ because $u$ and $\alpha^m(t_0^{-1})$ do, and}
\alpha^n(y) &= \alpha^n(u_0)t^{-1} = s,
\end{align*}
which belongs to $V_-$.

Setting $z = y^{-1}w$. We have for $m\leq k\leq n$ that 
$$
\alpha^k(z) = \alpha^k(t_0) \alpha^k(u_0^{-1}) \alpha^{k-m}(uv) = \alpha^k(t_0) \alpha^{k-m}(v).
$$
Then $\alpha^k(t_0)\in V_+$, because $\alpha^n(t_0)\in V_+$ and $k\leq n$, and $\alpha^{k-m}(v)\in V_-$, because $v\in V_-$ and $k-m\geq0$, and it follows that $\alpha^k(z)\in V$. 
\end{proof}
The next two lemmas are standard facts that are used in the proof that it is necessary for $V_{++}$ to be closed in order for bounded $\alpha$-orbits to be contained in~$V$. 
\begin{lemma}[Bourbaki]
\label{lem:bourbaki} 
Let $G$ be a topological group.  Let
$H$ be a subgroup of $G$ and ${\mathcal U}$ a neighbourhood of
$e$.  Then
$H$ is closed if and only if $H \cap {\mathcal U}$ is relatively
closed.
\end{lemma}
\begin{proof} 
If $H$ is closed, then $H \cap {\mathcal U}$ is relatively
closed.

Suppose that $H$ is not closed.  Then there is an $x \in
H^-\backslash H$. Now ${\mathcal U}^{-1}x$ is a neighbourhood
of $x$ and so there is an $h \in {\mathcal U}^{-1}x$.  Then
$xh^{-1} \in {\mathcal U}$ and also $xh^{-1} \in H^-\backslash
H$.  Hence $H \cap {\mathcal U}$ is not relatively closed. 
\end{proof}

\begin{lemma}
\label{lem:fact}
Let $K$ be a compact group and suppose that $K_n$, $n\in\mathbb{N}$ are closed subgroups of $K$ with $K_n\leq K_{n+1}$ for each $n$ and $\bigcup_{n\in\mathbb{N}} K_n = K$. Then there is an~$n$ such that $K_n = K$. 
\end{lemma}
\begin{proof}
Denote the Haar measure on $K$ by $\mu$. Then $\lim_{n\to\infty} \mu(K_n) = \mu(K)$ and so there is an $m$ such that $\mu(K_m)>0$. Hence $K_m$ is open and $\{K_n\}_{n\geq m}$ is an increasing sequence of open subgroups with union equal to $K$. Therefore $K_n = K$ for some $n$. 
\end{proof}

Here, finally, are the criteria for $L_V$ to be contained in $V$. 
\begin{proposition}
\label{prop:Vplusplus}
Let $V$ be tidy above for the endomorphism $\alpha$. Then the following are equivalent:
\begin{enumerate}
\item\label{Vplusplus1} $L_V\leq V$;
\item\label{Vplusplus2} if $\alpha^m(v)\in V$ and $\alpha^n(v)\in V$ for some $0\leq m\leq n$, then $\alpha^k(v)\in V$ for every $m\leq k\leq n$;
\item\label{Vplusplus3} $V_{++} \cap V = V_+$ and $K_V\leq V$; and 
\item\label{Vplusplus4} $V_{++}$ is closed and $\left\{ [\alpha^{n+1}(V_+) : \alpha^n(V_+)]\right\}_{n\in\mathbb{N}}$ is constant.
\end{enumerate}
If these conditions are satisfied, then $L_V = V_+\cap V_-$. 
\end{proposition}

\begin{proof} $(\ref{Vplusplus1})\Rightarrow(\ref{Vplusplus2})$ Suppose that (\ref{Vplusplus2}) fails, so that  there is $v\in G$ with $\alpha^m(v)\in
V$, $\alpha^n(v)\in V$ and  $\alpha^k(v)\not\in V$ for some $m\leq k\leq
n$. Then, by Lemma~\ref{lem:cut}, there is a $y \in G$ such
that $\alpha^m(y)\in V_+$ and
$\alpha^n(y)\in V_-$ but $\alpha^k (y) \not\in V$. Hence
$\alpha^k(y)$ belongs to ${\mathcal L}_V\setminus V$ and (\ref{Vplusplus1}) fails to hold.

$(\ref{Vplusplus2}) \Rightarrow (\ref{Vplusplus3})$ Suppose that
(\ref{Vplusplus2}) holds. Let $w\in V_{++}
\cap V$, so that $w= \alpha^n(v)$ for some $v\in
V_+$ and $w\in V$. Then (\ref{Vplusplus2}) implies that $\alpha^k(v)\in V$ for every $0\leq k\leq n$. Hence $w\in V_+$. Therefore $V_{++} \cap V \subset V_+$. The reverse inclusion is
immediate. Similarly, if $w$ belongs to $\mathcal{K}_V$, there are $m\leq n\in \mathbb{N}$ and $v\in V_+$ such that $w = \alpha^m(v)$ and $\alpha^n(v) = \ident$. Then (\ref{Vplusplus2}) implies that $w\in V$ and it has been shown that $\mathcal{K}_V\leq V$. Then $K_V\leq V$ because $V$ is closed. 

$(\ref{Vplusplus3})\Rightarrow(\ref{Vplusplus4})$ Assuming that (\ref{Vplusplus3}) holds, $V_{++} \cap V$ is closed because it is equal to~$V_+$. Then Lemma~\ref{lem:bourbaki} implies that $V_{++}$ is closed. Furthermore, we have that $K_V\leq V$ and it follows from Proposition~\ref{prop:TB_criterion1} that $[\alpha^{n+1}(V_+) : \alpha^n(V_+)]$ does not depend on~$n$. 

$(\ref{Vplusplus4})\Rightarrow(\ref{Vplusplus1})$ Assume that
(\ref{Vplusplus1}) fails and that $V_{++}$ is closed. Then closedness of $V_{++}$ implies that $L_V\leq V_{++}$, whence $\{\alpha^n(V_+)\}_{n\in\mathbb{N}}$ is an increasing sequence of closed subgroups of $V_{++}$ that covers $L_V$. Hence, $\{\alpha^n(V_+)\cap L_V\}_{n\in\mathbb{N}}$ is an increasing sequence of closed subgroups of $L_V$ whose union is equal to $L_V$ and there is, by Lemma~\ref{lem:fact}, an $n$ such that $\alpha^n(V_+)\cap L_V = L_V$. Choose $n$ to be the smallest value for which this holds. Then $n\geq1$ because $L_V\not\leq V$ by assumption and $\left(\alpha^n(V_+)\setminus \alpha^{n-1}(V_+)\right)\cap L_V$ is not empty. Let $x\in \left(\alpha^n(V_+)\setminus \alpha^{n-1}(V_+)\right)\cap L_V$. Then $\alpha(x)$ is in $L_V$ because $x$ is and so, since $\alpha^n(V_+)$ covers $L_V$, there is $y\in \alpha^{n-1}(V_+)$ such that $\alpha(y) = \alpha(x)$. Since $xy^{-1}$ is in $V_{++}$ and $\alpha(xy^{-1}) = \ident$, $xy^{-1}$ belongs to $K_V$ by definition. Since $x\in \alpha^n(V_+)\setminus \alpha^{n-1}(V_+)$ and $y\in \alpha^{n-1}(V_+)$, it follows that $xy^{-1}$ belongs to $\left(\alpha^n(V_+)\setminus \alpha^{n-1}(V_+)\right)\cap K_V$. In particular, $xy^{-1}\in K_V\setminus V_+$ and so, although $xy^{-1}$ itself might belong to $V$, $xy^{-1} = \alpha^m(z)$ for some $m\in\mathbb{N}$ and $z\in K_V\setminus V$. Hence $\left\{ [\alpha^{n+1}(V_+) : \alpha^n(V_+)]\right\}_{n\in\mathbb{N}}$ is not constant, by Proposition~\ref{prop:TB_criterion1}, and~(\ref{Vplusplus4}) fails. 
\end{proof}

The criteria identified in Proposition~\ref{prop:Vplusplus} will be given a name.
\begin{definition}
\label{defn:tidy_below}
The compact, open subgroup $V$ is \emph{tidy below} for $\alpha$ in $\End{G}$ if:
\begin{description}
\item[TB1] $V_{++}$ is closed; and 
\item[TB2] the sequence $\left\{ [\alpha^{n+1}(V_+) : \alpha^n(V_+)]\right\}_{n\in\mathbb{N}}$ is constant.
\end{description}
\end{definition}

When $\alpha$ is an automorphism, it follows from results in~\cite{Wi:structure} that $V$ is tidy below for $\alpha$ if and only if it is tidy below for $\alpha^{-1}$, that is, $V_{--} := \bigcup_{n\geq0} \alpha^{-n}(V_-)$ is closed. When $\alpha$ is not an automorphism, the symmetry between $\alpha$ and $\alpha^{-1}$ can no longer be exploited and, indeed, it is no longer possible to define $V_{--}$ as a union of images under $\alpha^{-1}$. However it turns out that the definition of $V_{--}$ as a union of inverse images makes sense and that tidiness below can be characterized in terms of it. 
\begin{proposition}
\label{prop:inV}
Let $V$ be a compact, open subgroup of $G$ that is tidy above for the endomorphism $\alpha$. Then $V$ is tidy below for $\alpha$ if and only if 
$$
V_{--} := \bigcup_{n\in\mathbb{N}} \alpha^{-n}(V_-)
$$ 
is closed. 
\end{proposition}
\begin{remark}
Since $V_-$ is a closed subgroup of $G$, $V_{--}$ is an increasing union of closed subgroups and is therefore a subgroup. 
\end{remark}
\begin{proof}
Assume that $V$ is tidy below and consider $v\in V_{--}\cap V$. Then, by definition, there is $n\geq0$ such that $\alpha^n(v)$ belongs to $V_-$ and it follows by Proposition~\ref{prop:Vplusplus}(\ref{Vplusplus2}) that $\alpha^k(v)\in V$ for every $k\in \{0,1,\dots, n\}$. Hence $v\in V_-$ and we have shown that  $V_{--}\cap V = V_-$, which is closed. Therefore, by Lemma~\ref{lem:bourbaki}, $V_{--}$ is closed. 

Assume that $V_{--}$ is closed. Then, since $\mathcal{L}_V\leq V_{--}$ by definition, we have $L_V\leq V_{--}$ and $\{\alpha^{-n}(V_-)\cap L_V\}_{n\in\mathbb{N}}$ is an increasing sequence of closed subgroups that covers $L_V$. Hence, by Lemma~\ref{lem:fact}, there is an $n$ such that $\alpha^{-n}(V_-)\cap L_V = L_V$, so that $\alpha^n(L_V)\leq V_-$. Since $L_V$ is $\alpha$-stable, it follows that $L_V\leq V_-$. In particular, $L_V\leq V$ and $V$ is tidy below by Proposition~\ref{prop:Vplusplus}.
\end{proof}

\section{Minimizing subgroups are tidy and conversely}
\label{sec:Min_equiv_Tidy}

It has been seen in previous sections that the compact, open subgroup $V$ is not minimizing for $\alpha$ if it is not tidy above for $\alpha$ or if it does not contain $L_V$, that is, is not tidy below for $\alpha$. In this section it is shown that tidiness above and below suffice to ensure that $V$ is minimizing. Hence $V$ is minimizing for $\alpha$ if and only if it is tidy for $\alpha$ as now defined.  
\begin{definition}
\label{defn:tidy}
The compact, open subgroup $V$ is \emph{tidy} for $\alpha$ in $\End{G}$ if it is both tidy above and tidy below for $\alpha$. 
\end{definition}
The proof goes by reviewing the arguments of the previous sections to show that every minimizing subgroup is tidy. Then it is shown, in Proposition~\ref{prop:scalesame}, that all tidy subgroups have the same displacement index $[\alpha(V) : \alpha(V)\cap V]$, which must therefore be the minimum. An important step is to show that the intersection of two subgroups tidy for $\alpha$ is tidy, see Proposition~\ref{prop:intersection}. 

Let $\alpha$ be an endomorphism of $G$. The following procedure takes a
given compact open subgroup, $U$, of $G$ and modifies it to produce a compact, open subgroup that is tidy for $\alpha$. 
\begin{description}
\item[Step 1] Denote $U_{-n} = \bigcap_{k=0}^n
\alpha^{-k}(U)$. Then $U_{-n}$ is tidy above for all $n$ sufficiently large by, Proposition~\ref{prop:propertyTA}. Let $N$ be the first integer such that $U_{-N}$ is tidy above and put $V= U_{-N}$. Then
$$
\ind{\alpha(V)}{\alpha(V)\cap V} \leq \ind{\alpha(U)}{\alpha(U)\cap U},
$$
with equality if and only if $U$ is tidy above, in which case $V=U$.
\item[Step 2] For this $V$ define ${\mathcal L}_V$ and $L_V$ as in Definition~\ref{de:curlyL}:
$$
\mathcal{L}_V = \left\{ x \in G : \exists y\in V_+\text{ and } m,n\in
{\mathbb{N}}\text{ with } \alpha^{m}(y)=x 
 \text{ and }\alpha^n (x) \in V_-  \right\},
 $$
 and $L_V = \overline{\mathcal{L}_V}$.\\
Proposition~\ref{prop:Liscompact} shows that ${ L}_V$ is a compact, $\alpha$-stable subgroup of $G$.
\item[Step 3] Define $\widetilde{V} := \left\{ x\in V\mid xL_V\subseteq L_VV\right\}$ and put $W = \widetilde{V}L_V$.  Then $W$ is a compact, open subgroup of $G$ by Lemma~\ref{lem:commute_groups}, is tidy above for $\alpha$ by Proposition~\ref{prop:define_W} and is tidy below by Proposition~\ref{prop:Vplusplus}. Proposition~\ref{prop:define_W} further shows that 
$$
\ind{\alpha(W)}{\alpha(W)\cap W} \leq \ind{\alpha(V)}{\alpha(V)\cap V}
$$
with equality if and only if $V$ is tidy below, in which case
$W=V$.
\end{description} 

This tidying procedure implies the following. 
\begin{proposition}
\label{prop:min=>tidy}
Every compact, open subgroup of $G$ that is minimizing for $\alpha$ in $\End{G}$ is tidy for $\alpha$. There exist compact, open subgroups of $G$ that are tidy for $\alpha$. \endproof 
\end{proposition}

As stated above, the first step of the proof of the converse to Proposition~\ref{prop:min=>tidy} is to show that the intersection of two tidy subgroups is tidy. A couple of technical results are needed for this. The first is a reformulation of Lemmas~\ref{lem:criterion} and~\ref{lem:criterion2} for tidy subgroups.
\begin{proposition}
\label{prop:winWplus}
Let $\alpha\in\End{G}$ and suppose that $V$ is tidy for $\alpha$. 
\begin{enumerate}
\item If $v\in V$ has an $\alpha$-regressive sequence that has an accumulation point, then $v$ belongs to $V_+$. 
\label{lem:winWplus1}
\item If $\{\alpha^{n}(v)\}_{n\in\mathbb{N}}$ has an accumulation point, then $v$ belongs to $V_-$. 
\label{lem:winWplus2}
\end{enumerate}
\end{proposition}

\begin{proof}
(\ref{lem:winWplus1}) Since $V$ satisfies \TA, $v = v_+v_-$, where $v_\pm\in V_\pm$. Then $v_- = v_+^{-1}v$ has an $\alpha$-regressive sequence with an accumulation point
and it follows, by Lemma~\ref{lem:criterion2}, that $v_-$ is in $L_V$. Since $V$ satisfies \TB, $L_V = V_+\cap V_-$ and $v$ belongs to $V_+$. 

(\ref{lem:winWplus2})
A similar argument that appeals to Lemma~\ref{lem:criterion} establishes the claim.
\end{proof}

\begin{lemma}
\label{lem:intersect_plus}
Let $W^{(1)}$ and $W^{(2)}$ be compact open subgroups of $G$ that are tidy for the endomorphism $\alpha$.  Then 
$$
W_+^{(1)} \cap W_+^{(2)} = \left(W^{(1)}\cap W^{(2)}\right)_+\text{ and }W_-^{(1)} \cap W_-^{(2)} = \left(W^{(1)}\cap W^{(2)}\right)_-.
$$
\end{lemma}
\begin{proof}
That $W_+^{(1)} \cap W_+^{(2)} \geq (W^{(1)}\cap W^{(2)})_+$ is clear. For the reverse inclusion, consider $w\in W_+^{(1)} \cap W_+^{(2)}$. Then $w = \alpha(v_i)$ with $v_i\in W_+^{(i)}$ for $i=1,2$. The element $v_1v_2^{-1}$ has a bounded $\alpha$-regressive sequence and satisfies $\alpha(v_1v_2^{-1}) = \ident$ whence, by Corollary~\ref{cor:bounded_ in_LV}, it belongs to $L_{W^{(2)}}$. Therefore $v_1 = (v_1v_2^{-1})v_2$ belongs to $W_+^{(2)}$ as well as to $W_+^{(1)}$. It has thus been shown that, for each $w\in W_+^{(1)} \cap W_+^{(2)}$, $w = \alpha(v_1)$ where $v_1\in W_+^{(1)} \cap W_+^{(2)}$. Induction  produces an $\alpha$-regressive sequence for $w$ that is contained in $W_+^{(1)} \cap W_+^{(2)}$, and it follows that $w$ belongs to $(W^{(1)} \cap W^{(2)})_+$. 

The equality $W_-^{(1)} \cap W_-^{(2)} = (W^{(1)}\cap W^{(2)})_-$ follows immediately from the definitions. 
\end{proof}

\begin{proposition}
\label{prop:intersection} 
Let $W^{(1)}$ and $W^{(2)}$ be compact open subgroups of $G$ that are tidy for the endomorphism $\alpha$.  Then $W^{(1)} \cap W^{(2)}$
is tidy for $\alpha$.
\end{proposition}

\begin{proof} Much of the argument involves establishing that the
intersection satisfies \TA\ for $\alpha$. Let $w$ be in $W^{(1)}
\cap W^{(2)}$.  Then $w = w_+w_-$, where $w_\pm \in W_\pm^{(1)}$, because $W^{(1)}$ satisfies \TA. However $w_+$ and $w_-$ need not be in $W^{(2)}$ even though $w$ is. It will be shown that there is $x$ in $W_+^{(1)} \cap W_-^{(1)}$ such that 
$$
w_+x^{-1} \in (W^{(1)} \cap W^{(2)})_+\text{ and }xw_- \in (W^{(1)} \cap W^{(2)})_-.
$$

Since $w_+$ belongs to $W_+^{(1)}$, there is an $\alpha$-regressive sequence $\{w_n\}_{n\in \mathbb{N}}\subset W_+^{(1)}$ for $w_+$. Let $u$ be an accumulation point of $\{w_n\}_{n\in \mathbb{N}}$. Then $u \in W_+^{(1)} \cap W_-^{(1)}$, by Lemma
\ref{lem:accpoint}.  Choose an integer, $N$, such that $w_{N} \in W^{(2)}u$. Then $w_{N} u^{-1}$ belongs to $W_+^{(1)}\cap  W^{(2)}$. Since $w_{N} u^{-1}$ is in $W_+^{(1)}$, it has a bounded $\alpha$-regressive sequence, $\{y_n\}_{n\in\mathbb{N}}$.  Then, since $w_{N} u^{-1}$ is in $W^{(2)}$ which is tidy, $w_{N}u^{-1}$ belongs to $W^{(2)}_+$ by Proposition~\ref{prop:winWplus}. Put $x = \alpha^{N} (u)$, which belongs to $W_+^{(1)}\cap W_-^{(1)}$. Then it has been shown that 
\begin{equation}
\label{eq:Wone}
w_+x^{-1} = \alpha^{N}(w_{N}u^{-1}) \in W_+^{(1)}\text{ and }
w_{N}u^{-1} \in W^{(1)}_+ \cap W^{(2)}_+.
\end{equation}

It is yet to be shown that $w_+x^{-1}$ is in $\left (W^{(1)} \cap W^{(2)}\right )_+$. To this end, use that $W^{(2)}$ satisfies
\TA\ for $\alpha$ to write $w = \hat{w}_\pm$ with $\hat{w}_\pm \in W_\pm^{(2)}$. Then there is an $\alpha$-regressive sequence $\{\hat{w}_n\}_{n\in\mathbb{N}}$ for $\hat{w}_+$ in $W_+^{(2)}$, and the argument of the previous paragraph may be repeated, possibly increasing the value of $N$, to find $v, y = \alpha^N(v)$ in $W_+^{(2)} \cap W_-^{(2)}$ such that 
\begin{equation}
\label{eq:Wtwo}
\hat{w}_+y^{-1} = \alpha^{N}(\hat{w}_{N}v^{-1}) \in W_+^{(2)}\text{ and }
\hat{w}_{N}v^{-1} \in W^{(1)}_+ \cap W^{(2)}_+.\end{equation} 

Then (\ref{eq:Wone}) and (\ref{eq:Wtwo}) imply that 
\begin{equation}
\label{eq:Wmeet}
v\hat{w}_N^{-1}w_Nu^{-1} \in W^{(1)}
\cap W^{(2)}.
\end{equation} 
Since $w_+w_- = \hat{w}_+\hat{w}_- = w$, we have
$$
\hat{w}_+^{-1}w_+ = \hat{w}_- {w}_-^{-1}\in \left( W^{(2)}_+W^{(1)}_+ \right)
\cap \left( W^{(2)}_-W^{(1)}_- \right).
$$
Then there is a bounded $\alpha$-regressive sequence for $\hat{w}_+^{-1}w_+$ and $\left\{\alpha^n(\hat{w}_+^{-1}w_+)\right\}_{n\in\mathbb{N}}$ is bounded. Hence there is a bounded $\alpha$-orbit, $\{s_n\}_{n\in\mathbb{N}}$ say, with $s_0 = \hat{w}_+^{-1}w_+$. Similiarly, since $u$ and $v$ belong to $W_+^{(1)}\cap W_-^{(1)}$ and $W_+^{(2)}\cap W_-^{(2)}$ respectively there are bounded $\alpha$-orbits $\{u_n\}_{n\in\mathbb{N}}$ and $\{v_n\}_{n\in\mathbb{N}}$ with $u_0 = u$ and $v_0 = v$. Hence  
$v\hat{w}_N^{-1}w_Nu^{-1}$ lies on a bounded $\alpha$-orbit and also lies in $W^{(1)}\cap W^{(2)}$ by~\eqref{eq:Wmeet}. Then applying Proposition~\ref{prop:bounded_ in_LV} or Proposition~\ref{prop:winWplus}  to $W^{(1)}$ and $W^{(2)}$ separately yields that
$$
v\hat{w}_N^{-1}w_Nu^{-1} \text{ belongs to }
\left(W_+^{(1)} \cap W_-^{(1)}\right) \cap \left( W_+^{(2)}\cap W_-^{(2)}\right).
$$ 
Hence $\alpha^N(v\hat{w}_N^{-1}w_Nu^{-1}) = y\hat{w}_+^{-1}w_+x^{-1}$ belongs to $(W_+^{(1)} \cap W_-^{(1)}) \cap ( W_+^{(2)}\cap W_-^{(2)})$ which, together with
the identity
$$
w_+ x^{-1} = (\hat{w}y^{-1})(y\hat{w}_+^{-1}w_+x^{-1}),
$$
(\ref{eq:Wone}), (\ref{eq:Wtwo}) and Lemma~\ref{lem:intersect_plus}, yields that
$$
w_+x^{-1} \in W^{(1)}_+\cap W^{(2)}_+ = \left(W^{(1)}\cap
W^{(2)}\right)_+.
$$

Turning to $xw_-$, the identity $xw_- = (w_+x^{-1})^{-1} w_+w_-$, where $w_+x^{-1}$ and  $w_+w_-$  belong to $W^{(1)}\cap W^{(2)}$, shows that $xw_-$ is in $W^{(1)}\cap
W^{(2)}$ as well. Then, since $\left \{ \alpha^n (xw_-)
\right
\}_{n \geq 0}$  has compact closure, Proposition~\ref{prop:winWplus}
applied to $W^{(1)}$ and $W^{(2)}$ separately, together with Lemma~\ref{lem:intersect_plus}, shows that 
$$
xw_- \in W_-^{(1)} \cap W^{(2)}_- = \left (W^{(1)} \cap
W^{(2)}\right )_-.
$$ 
Therefore $w = (w_+x^{-1})(xw_-)$ is the product of an element of
$\left (W^{(1)} \cap W^{(2)}\right )_+$ and an element of $\left(W^{(1)}
\cap W^{(2)} \right)_-$ and $W^{(1)} \cap W^{(2)}$ satisfies \TA.

To see that $W^{(1)} \cap W^{(2)}$ satisfies \TB1, let $y$ be in the closure of $\left (W^{(1)} \cap W^{(2)}\right)_{++}$. Then $y$ is in the closure of $W^{(i)}_{++}$ for $i = 1, 2$ and, since both $W^{(1)}$ and $W^{(2)}$ are tidy below, $y$ belongs to both $W^{(1)}_{++}$ and $W^{(2)}_{++}$. Hence there are $n \geq 0$ and $z_1$ and $z_2$ in  $W_+^{(1)}$ and $W_+^{(2)}$ respectively such that $\alpha^{n} (z_1) = y = \alpha^n(z_2)$. Then $z_1z_2^{-1}$ has a bounded $\alpha$-regressive sequence and $\alpha^n(z_1z_2^{-1}) = \ident$. Hence, by Corollary~\ref{cor:bounded_ in_LV}, $z_1z_2^{-1}$ is in $L_{W^{(2)}}$, whence $z_1 = (z_1z_2^{-1})z_2$ is in $W_+^{(1)}\cap W_+^{(2)}$  and $y$ belongs to $(W^{(1)} \cap W^{(2)})_{++}$. Therefore $(W^{(1)} \cap W^{(2)})_{++}$ is closed. To see that $W^{(1)} \cap W^{(2)}$ satisfies \TB2, note that since $W^{(1)}$ and $W^{(2)}$ are tidy below, $K_{W^{(1)}\cap W^{(2)}}\leq W^{(i)}$ for $i=1,2$ by Corollary~\ref{cor:bounded_ in_LV}. Therefore $K_{W^{(1)}\cap W^{(2)}}\leq W^{(1)}\cap W^{(2)}$. 
\end{proof}

That the displacement index is the same for all subgroups tidy for $\alpha$, which may now be proved, is  the last step in the proof that tidy subgroups are minimizing. 
\begin{proposition}
\label{prop:scalesame} 
 Let $\alpha\in\End{G}$ and let $W^{(1)}$ and $W^{(2)}$ be tidy for $\alpha$. Then
$$
\left[ \alpha (W^{(1)}) : \alpha (W^{(1)}) \cap W^{(1)} \right
] =  \left [ \alpha (W^{(2)}) : \alpha (W^{(2)}) \cap W^{(2)}
\right]. 
$$
\end{proposition}

\begin{proof} 
Since $W^{(1)}$ and $W^{(2)}$ satisfy \TA, it suffices by Lemma
\ref{lem:indexreduced} to show that 
\begin{equation}
\label{eq:compareindex}
\left[ \alpha (W^{(1)}_+) : W^{(1)}_+
\right ] = \left [ \alpha (W^{(2)}_+) : W^{(2)}_+ \right ].
\end{equation} 
Since, by Proposition~\ref{prop:intersection}, $W^{(1)}\cap W^{(2)}$ is tidy, equality may be proved by showing that both
sides of (\ref{eq:compareindex}) equal $\left [ \alpha ((W^{(1)}\cap
W^{(2)})_+) : (W^{(1)}\cap W^{(2)})_+ \right ]$. Hence it suffices to assume that
$W^{(2)} \leq W^{(1)}$.  

 The chain
of inclusions $W^{(2)}_+ \leq W^{(1)}_+ \leq \alpha
(W^{(1)}_+)$ implies that
\begin{align*}
\label{eq:firstchain}
\left[ \alpha (W^{(1)}_+) : W^{(2)}_+ \right] &= \left[ \alpha
(W^{(1)}_+) : W^{(1)}_+ \right] \left[ W^{(1)}_+ : W^{(2)}_+
\right],\\
\intertext{while the chain $W^{(2)}_+ \leq \alpha(W^{(2)}_+) \leq \alpha
(W^{(1)}_+)$ implies that}
\left[ \alpha (W^{(1)}_+) : W^{(2)}_+ \right] &= \left[ \alpha
(W^{(1)}_+) : \alpha (W^{(2)}_+) \right] \left[ \alpha( W^{(2)}_+)
: W^{(2)}_+ \right].
\end{align*}
The proof of (\ref{eq:compareindex}) will be completed by showing that 
$$
\left[ \alpha (W^{(1)}_+) : \alpha(W^{(2)}_+) \right] = \left[  W^{(1)}_+ :W^{(2)}_+ \right],
$$
which is clear when $\alpha$ is an automorphism but requires an additional argument for general endomorphisms. For this, consider the well-defined and onto map  
$$
\phi : W^{(1)}_+/W^{(2)}_+ \to \alpha(W^{(1)}_+)/\alpha(W^{(2)}_+) \text{ given by } \phi(wW^{(2)}_+) = \alpha(w)\alpha(W_+^{(2)}) .
$$ 
If $w\in W_+^{(1)}$ and $\alpha(w)\in \alpha(W_+^{(2)})$, then there is
$v\in W_+^{(2)}$ such that $\alpha(v) = \alpha(w)$. Then $wv^{-1}\in \ker(\alpha)$ and has a bounded $\alpha$-regressive sequence. Hence, by Corollary~\ref{cor:bounded_ in_LV}, $wv^{-1}$ is in $L_{W^{(2)}}$, whence $w$ belongs to 
to $W_+^{(2)}$ and $\phi$ is one-to-one as is required to establish the claim. 
\end{proof}

As discussed at the beginning of this section,  every subgroup that is minimizing for $\alpha$ is tidy. It follows from the Proposition~\ref{prop:scalesame} that, if one tidy subgroup is minimizing, then all are. Hence tidiness characterizes minimizing subgroups. 
\begin{theorem}[The Structure of Minimizing Subgroups]
\label{thm:tidystucture}
Let $\alpha$ be an endomorphism of the totally disconnected, locally compact group ${G}$. Then the compact open subgroup $W$ of $G$ is minimising for $\alpha$ if and only if it is tidy for $\alpha$. 
\endproof
\end{theorem}

When $\alpha$ is automorphism, it is clear that $\alpha^n(W)$ is minimizing for $\alpha$, and hence tidy, when $W$ is. It is then an immediate consequence of Proposition~\ref{prop:intersection} that, if $W$ is tidy for $\alpha$, then so are the subgroups $\alpha^m(W)\cap \alpha^n(W)$ for every $m\leq n\in\mathbb{Z}$. Given one tidy subgroup it is thus possible to find many others (unless $W$ is stable under $\alpha$) and it follows in particular that $W_+\cap W_- =\bigcap_{n\in\mathbb{Z}} \alpha^n(W)$ is the intersection of tidy subgroups. The same is not true for endomorphisms in general because $\alpha^n(W)$ need not be compact when $n$ is negative or open when $n$ is positive. We conclude this section with a couple of partial results that build new tidy subgroups.
\begin{proposition}
\label{prop:alphan}
Let $W$ be tidy for $\alpha\in \End{G}$. Define compact, open subgroups $W_{[n]}$, $n\in\mathbb{N}$, recursively by setting $W_{[0]} = W$ and
$$
W_{[n+1]} = \left\{ x\in W_{[n]}\mid x\alpha^n(W_+) \subset \alpha^n(W_+)W_{[n]}\right\} 
$$
\ and set $W^{[\alpha,n]} = \alpha^n(W_+)W_{[n]}$. Then $W_{[n]}$ and $W^{[\alpha,n]}$
are tidy for $\alpha$ and $\alpha^n(W_+)$ is contained in $W^{[\alpha,n]}$ for each $n\in\mathbb{N}$.  
\end{proposition}
\begin{proof}
The proof is by induction on $n$. All claims clearly hold when $n=0$, for $W_{[0]} = W = W^{[\alpha,0]}$. It will be seen that the inductive step is equivalent to the $n=1$ case. 

Lemma~\ref{lem:commute_groups} implies that $W_{[1]}$ is an open subgroup of $W$, and it is clear that $W_+\leq W_{[1]}$. Then, since $W$ is tidy above, 
$$
W_{[1]} = W_+(W_-\cap W_{[1]})
$$
It will be shown that $W_-\cap W_{[1]}$ is invariant under $\alpha$, from which it follows that $W_-\cap W_{[1]} \leq W_{[1]-}$ and thence that $W_{[1]}$ is tidy above. For this, let $x\in W_-\cap W_{[1]}$ and $l\in \alpha(W_+)$. Then 
\begin{eqnarray*}
\alpha(x)l &=& \alpha(xl'), \text{where } l'\in W_+,\\
&=& \alpha(l''x'), \text{ with }x'\in W_-,\ l''\in W_+\text{ because }W\text{ is tidy above,}\\
&=& \alpha(l'')\alpha(x') \in \alpha(W_+)W_-,
\end{eqnarray*} 
which shows that $\alpha(x)$ is in $W_-\cap W_{[1]}$. Property \TB\ holds for $W_{[1]}$ because it contains $L_W = W_+\cap W_-$.
Hence $W_{[1]}$ is tidy for $\alpha$. Appealing to Lemma~\ref{lem:commute_groups} again, $W^{[\alpha,1]} = \alpha(W_+)W_{[1]}$ is a compact, open subgroup of $G$, and $W^{[\alpha,1]}$ is tidy for~$\alpha$ because, by construction, it is equal to $\alpha(W_+)(W_-\cap W_{[1]})$ and contains $L_W$. 

The induction proceeds by arguing with $W^{[\alpha,n]}$ in place of $W$ and noting that $W^{[\alpha,n]}_+ = \alpha^n(W_+)$.
\end{proof}

It is immediate from the definition that $\{W_{[n]}\}_{n\in\mathbb{N}}$ is a non-increasing sequence of tidy subgroups of $W$. The sequence need not be decreasing however because, for example, if $G$ is abelian then it is constant. However, tidiness of certain proper subgroups of $W$ that contain $W_+$ may be established with the aid of the following lemma.

\begin{lemma}
\label{lem:commute_compact_groups}
Let $K$ and $L$ be compact subgroups of $G$ and suppose that $KL = LK$. Let $K'$ be an open subgroup of $K$ and put
$$
\hat{K} = \left\{ x\in K' \mid xL\subseteq LK'\right\}.
$$
Then $\hat{K}$ is an open subgroup of $K'$ and $\hat{K}L = L\hat{K}$. 
\end{lemma}
\begin{proof}
That $\hat{K}$ is a closed semigroup and therefore a subgroup of the compact group $K'$ follows exactly as in the proof of Lemma~\ref{lem:commute_groups}, as does the proof that $\hat{K}L = L\hat{K}$. 

To see that $\hat{K}$ is open in $K'$, note that $L$ normalises small open subgroups of $G$ because it is compact. Hence there is a compact, open subgroup, $V$, of $G$ that is normalised by $L$ and such that $V\cap K\leq K'$. Then, for $x\in V\cap K$ and $l\in L$, $xl = lx'$ where $x'\in V\cap K$. Therefore the open subgroup $V\cap K$ is contained in $\hat{K}$ and $\hat{K}$ is open.  
\end{proof}

\begin{proposition}
\label{prop:tidy_subgroups}
Let $W$ be a compact, open subgroup of $G$ that is tidy for $\alpha$ in $\End{G}$. Then:
\begin{enumerate}
\item $W_{-n}$ is tidy for every $n\in\mathbb{N}$; and 
\label{prop:tidy_subgroups1}
\item for every open subgroup, $V$, of $W_-$ that contains $W_+\cap W_-$, there is an open subgroup, $\hat{W}$, of $W$ that is tidy for $\alpha$, contains $W_+$ and satisfies $\hat{W}\cap W_- \leq V$.
\label{prop:tidy_subgroups2}
\end{enumerate}
\end{proposition}
\begin{proof}
(\ref{prop:tidy_subgroups1}) Since, by Lemma~\ref{lem:stable}, $[W_{-n} : W_{-n-1}] = [W : W_{-1}] = [\alpha(W) : \alpha(W)\cap W]$ for each $n$ when $W$ is tidy above for $\alpha$, $W_{-n}$ is minimizing, and hence tidy, for every $n\in\mathbb{N}$.

(\ref{prop:tidy_subgroups2}) It will be important later to know that $\alpha(V)\leq V$, which may be achieved by passing to a subgroup if necessary for, by Proposition~\ref{prop:propertyTA}, $V$ may be replaced by an open subgroup that is tidy above for the endomorphism $\alpha|_{W_-}$. Then, since $V_+ = W_+\cap W_-$ is $\alpha$-stable, $V = V_-$ and $\alpha(V)\leq V$ as required. 

Let $\hat{V} = \left\{ x\in V \mid xW_+\subseteq W_+V\right\}$. Then $\hat{V} \geq W_+\cap W_-$ and $\hat{V}$ is an open subgroup of $V$ such that $\hat{V}W_+ = W_+\hat{V}$, by Lemma~\ref{lem:commute_compact_groups}. Put $\hat{W} = W_+\hat{V}$. Then $\hat{W}$ is a closed subgroup of $W$ that has finite index and is therefore open. Moreover, $\hat{W}$ contains $W_+$ and $\hat{W}\cap W_- = \hat{V}\leq V$. 

By construction, $\hat{W} = W_+\hat{V}$ and $\hat{W}$ contains $W_+\cap W_- = L_W$. Since $W_+ = \hat{W}_+$ and $L_{\hat{W}} \leq L_W$ (in fact, they are equal), it follows from Proposition~\ref{prop:Vplusplus} that $\hat{W}$ is tidy for $\alpha$ provided that $\hat{V}\leq \hat{W}_-$. This may be established by recalling that it has been arranged that $\alpha(V)\leq V$ because, letting $x\in \hat{V}$ and $w\in W_+$, we have
\begin{eqnarray*}
\alpha(x)w &=& \alpha(xw'), \text{ for some }w'\in W_+\cap W_{-1},\\
&=& \alpha(w''x'), \text{ with }w''\in W_+\text{ and }x'\in V,\text{ because }x\in \hat{V}.
\end{eqnarray*}
Since $\alpha(x)w$ and $w''x'$ are in $W$, $w''x'$  belongs to $W_{-1}$ and so in fact $w''$ belongs to $W_+\cap W_{-1}$. Hence $\alpha(x)w$ is
in $W_+\alpha(V) \leq W_+V$,
and it has been shown that $\alpha(x)$ belongs to $\hat{V}$. Therefore $\alpha(\hat{V})\leq \hat{V}$ and it follows that $\hat{V}$ is contained in $\hat{W}_-$. 
\end{proof}
\begin{corollary}
\label{cor:tidy_subgroups}
Let $W$ be tidy for the endomorphism $\alpha$ of $G$. Then 
$$
W_+\cap W_- = \bigcap\left\{ \hat{W} \mid W_+\cap W_-\leq \hat{W} \leq W\text{ and }\hat{W}\text{ is tidy for }\alpha\right\}.
$$
\end{corollary}
\begin{proof}
From Proposition~\ref{prop:tidy_subgroups}(\ref{prop:tidy_subgroups1}) it follows that $W_-$ is contained in the intersection, while from Proposition~\ref{prop:tidy_subgroups}(\ref{prop:tidy_subgroups2}) it follows that $W_+$ is. 
\end{proof}

\section{Properties of the scale function on endomorphisms}
\label{sec:Moller}

The scale function $\alpha\mapsto s(\alpha) : \End{G} \to \mathbb{Z}^+$ defined here extends the function defined on automorphisms in~\cite{wi:further} (and that defined on inner automorphisms in~\cite{Wi:structure}). The scale is shown in~\cite{wi:further} to have certain algebraic properties as a function on the automorphism group of $G$ and these properties are shown in this section to extend to the endomorphism semigroup of $G$ as much as can be expected. (The only properties that do not extend are the characterization of $\alpha$-stable subgroups as those for which $s(\alpha) = 1 = s(\alpha^{-1})$ and, more generally, that $s(\alpha)/s(\alpha^{-1}) = \Delta(\alpha)$, the module of the automorphism $\alpha$.) 
\begin{proposition}
\label{prop:powers}
Let the compact open subgroup $W\leq G$ be tidy for the endomorphism $\alpha$. Then $W$ is tidy for $\alpha^k$ for every $k\in\mathbb{N}$. The converse is false even when $G$ is finite and $\alpha$ is an automorphism. 
\end{proposition}
\begin{proof}
The subgroups $W_n$, for $n\in\mathbb{N}$, are defined recursively in Definition~\ref{defn:Uplus} by $W_0 = W$ and $W_{n+1} = W\cap \alpha(W_n)$. Temporarily denoting these subgroups by $W_{\alpha,n}$ and the corresponding subgroups defined using $\alpha^k$ by $W_{\alpha^k,n}$,  it may be seen by induction that $W_{\alpha^k,n} \geq W_{\alpha,kn}$. Hence, extending the temporary notation, $W_{\alpha^k,+} \geq W_{\alpha,+}$. Similarly, and referring to Definition~\ref{defn:Uminus}, $W_{\alpha^k,-} \geq W_{\alpha,-}$.  Hence, since $W$ is tidy above for $\alpha$, it is tidy above for $\alpha^k$. 

Next, consider $w\in W_{\alpha^k,+}$. There is a sequence $\{w_{kn}\}_{n\in\mathbb{N}}\subset W$ that is $\alpha^k$-regressive for $w$, that is, $w_0 = w$ and $\alpha^k(w_{k(n+1)}) = w_{kn}$ for each $n$. Embed this in a sequence $\{w_n\}_{n\in\mathbb{N}}$ by interpolating
$$w_j = \alpha^{k(n+1)-j}(w_{k(n+1)}), \text{ for }kn < j < k(n+1).
$$ 
Then $\{w_n\}_{n\in\mathbb{N}}$ is $\alpha$-regressive for $w$ and, since $W$ is tidy for $\alpha$ and $w_{kn}\in W$ for each $n$, it follows from Proposition~\ref{prop:inV} that $w_n\in W$ for every $n\in\mathbb{N}$. Hence $w\in W_{\alpha,+}$ and it has been shown that $W_{\alpha^k,+} = W_{\alpha,+}$. That $W_{\alpha^k,-} = W_{\alpha,-}$ may be seen by a similar argument. 

Let $w\in \mathcal{L}_{\alpha^k,W}$, so that there are $v\in W_{\alpha^k,+}$ and $m\leq n$ such that $\alpha^m(v) = w$ and $\alpha^n(v)\in W_{\alpha^k,-}$. Then it follows from what has just been shown that $v\in W_{\alpha,+}$ and $\alpha^n(v) \in  W_{\alpha,-}$,whence $w\in \mathcal{L}_{\alpha,W}$. Since $W$ is tidy for $\alpha$, $\mathcal{L}_{\alpha,W}\leq W$. Hence $w\in W$ and it has been shown that $\mathcal{L}_{\alpha^k,W} \leq W$. Therefore $W$ is tidy for $\alpha^k$.

For an example where the converse fails, let $G$ be the finite group 
$$
C_p^2 = \left\{ (c_1,c_2) \mid c_i\in C_p\right\}
$$ 
and $\alpha$ be the automorphism $\alpha(c_1,c_2) = (c_2,c_1)$. Then the subgroup 
$$
W := \left\{ (\bar{0},c_2) \mid c_2\in C_p\right\}
$$ 
is tidy for $\alpha^2$, which is the identity automorphism, but not for $\alpha$.   
\end{proof}

Properties~{\bf S1} and~{\bf S2} in the next result correspond to the properties of the same name in~\cite{wi:further}. 
\begin{proposition}
\label{prop:scale}
The scale $s : \End{G} \to \mathbb{Z}^+$ satisfies:
\begin{description}
\item[S1] $s(\alpha) = 1$ if and only if there is  a compact, open $U\leq G$ with $\alpha(U)\leq U$;
\label{thm:scale1}
\item[S2] $s(\alpha^n) = s(\alpha)^n$ for every $n\in\mathbb{N}$.
\label{thm:scale2}
\end{description}
\end{proposition}
\begin{proof}
Property~{\bf S1} is immediate from the definition of the scale. 

For~{\bf S2}, let $W$ be tidy for $\alpha$. Then $W$ is tidy for $\alpha^n$ by Proposition~\ref{prop:powers} and $s(\alpha) = [\alpha(W_+) : W_+]$ and 
$$
s(\alpha^n) = [\alpha^n(W_+) : W_+] = \prod_{k=0}^{n-1} [\alpha^{k+1}(W_+) : \alpha^k(W_+)] = [\alpha(W_+) : W_+]^n
$$
because the sequence $\left\{ [\alpha^{k+1}(W_+) : \alpha^k(W_+)]\right\}$ is constant, by the definition of tidiness.
\end{proof}

R.\,M\"oller gave an alternative characterization of the scale in \cite[Theorem~7.7]{Moller}. The same characterization, analogous to a spectral radius formula, holds for endomorphisms.   
\begin{proposition}[M\"oller]
\label{defn:Moller's_formula}
Let $\alpha\in \End{G}$ and let $V$ be a compact, open subgroup of~$G$. Then 
$$
s(\alpha) = \lim_{n\to\infty} [\alpha^n(V) : V\cap \alpha^n(V)]^{\frac{1}{n}}.
$$ 
\end{proposition}
\begin{proof}
For $U$ and $V$ be compact, open subgroups of $G$ with $U\leq V$ we have
\begin{eqnarray*}
[\alpha^n(V) : U\cap \alpha^n(U)] &=& [\alpha^n(V) : \alpha^n(U)][\alpha^n(U) : U\cap \alpha^n(U)] \ \ \text{ and }\\     
{[}\alpha^n(V) : U\cap \alpha^n(U){]} &=& [\alpha^n(V) : V\cap \alpha^n(V)][V\cap \alpha^n(V) : U\cap \alpha^n(U)].
\end{eqnarray*}
Since $\alpha$ is an endomorphism, 
 $[\alpha^n(V) : \alpha^n(U)] \leq [V : U]$ and 
 $$
 [V\cap \alpha^n(V) : U\cap \alpha^n(U)] = [V\cap \alpha^n(V) : U\cap \alpha^n(V)][U\cap \alpha^n(V) : U\cap \alpha^n(U)]
 $$ 
 is at most $[V : U]^2$. Hence, and recalling too that all indices are at least~$1$, 
 $$
[V : U]^{-1}[\alpha^n(V) : V\cap \alpha^n(V)]\leq  [\alpha^n(U) : U\cap \alpha^n(U)]\leq [V : U]^2 [\alpha^n(V) : V\cap \alpha^n(V)].
$$ 
For $V$ and $W$ any two compact, open subgroups, putting $U = V\cap W$ in the above inequalities yields
$$
M_1[\alpha^n(W) : W\cap \alpha^n(W)]\leq  [\alpha^n(V) : V\cap \alpha^n(V)]\leq M_2 [\alpha^n(W) : W\cap \alpha^n(W)],
$$ 
where $M_1 = [V:U]^{-2}[W:U]$ and $M_2 = [V:U][W:U]^2$. 

Choose $W$ to be tidy for $\alpha$. Then $[\alpha^n(W) : W\cap \alpha^n(W)] = s(\alpha)^n$ for every $n\in\mathbb{N}$ and the claimed limit follows. 
\end{proof}

\section{Subgroups of $G$ associated with the endomorphism $\alpha$}
\label{sec:subgroups}

The subgroups $U_{++}$ and $U_{--}$ are defined in terms of a particular subgroup $U$ that is minimizing for the endomorphism $\alpha$. It will be seen next that they are contained in subgroups of $G$ defined directly in terms of $\alpha$, and the restriction of $\alpha$ to these subgroups will be described. These subgroups are familiar when $\alpha$ is an automorphism of a $p$-adic Lie group and are named by extension from that case.
\begin{definition}
\label{defn:para_etc}
Define the \emph{parabolic}, \emph{anti-parabolic} and \emph{Levi} subgroups respectively for the endomorphism $\alpha$ by:
\begin{align*}
\pbp{\alpha} &= \left\{ x\in G \mid \alpha^n(x) \text{ is bounded}\right\};\\
\pbn{\alpha} &= \left\{ x\in G \mid  \text{ there is a bounded $\alpha$-regressive sequence for }x\right\};\\
\text{ and }\quad\ml{\alpha} &= \pbp{\alpha}\cap \pbn{\alpha}.
\end{align*}
\end{definition} 
Note that $U_{--}\leq \pbp{\alpha}$ and $U_{++}\leq \pbn{\alpha}$ for any compact, open subgroup $U$. In particular, $V_{++}$ is contained in $\pbn{\alpha}$ for any subgroup minimizing for $\alpha$ and the scale of the restriction of $\alpha$ to $\pbn{\alpha}$ is equal to the scale of $\alpha$ on $G$. On the other hand, any compact, open subgroup of $\pbp{\alpha}$ that is tidy for $\alpha$ is invariant under $\alpha$ by Proposition~\ref{prop:winWplus}(\ref{lem:winWplus2}) and the scale of the restriction of $\alpha$ to $\pbp{\alpha}$ is~1. 

\begin{proposition}
\label{prop:par_closed}
Let $\alpha\in\End{G}$. Then $\pbp{\alpha}$, $\pbn{\alpha}$ and $\ml{\alpha}$ are closed $\alpha$-invariant subgroups of $G$.
\end{proposition}
\begin{proof}
That these subsets are subgroups of $G$ and are $\alpha$-invariant is immediate from their definitions.

To see that $\pbp{\alpha}$ is closed, consider $x$ in  $\overline{\pbp{\alpha}}$. Let $W$ be a compact, open subgroup of~$G$ that is tidy for $\alpha$, and choose $y,z\in xW\cap \pbp{\alpha}$. Then $y^{-1}z$ is in $W\cap\pbp{\alpha}$, whence $y^{-1}z$ belongs to $W_-$ by Proposition~\ref{prop:winWplus}(\ref{lem:winWplus2}). Hence $xW\cap \pbp{\alpha}$ is contained in $yW_-$, which is closed because $W_-$ is. Therefore $x$ is in $yW_-$ and  is consequently in $\pbp{\alpha}$. 

A similar argument appealing to Proposition~\ref{prop:winWplus}(\ref{lem:winWplus1}) shows that $\pbn{\alpha}$ is closed, and it follows immediately that $\ml{\alpha}$ is closed as well.  
\end{proof}

 It is clear that the restrictions of $\alpha$ to $\pbn{\alpha}$ and $\ml{\alpha}$ are onto homomorphisms and that the restriction to $\pbp{\alpha}$ need not be onto. Before analyzing these restrictions further, it is convenient to introduce some additional subgroups associated with $\alpha$.
\begin{definition}
\label{defn:iterated_kernel}
The \emph{bounded iterated kernel} of the endomorphism $\alpha$ is
$$
\ikk{\alpha} = \left\{ x\in G \mid \exists n \text{ such that }\alpha^n(x) = \ident\right\}^- \cap \pbn{\alpha},
$$
and the \emph{nub} of $\alpha$ is 
$$
\nub{\alpha} = \bigcap\left\{ W\leq G \mid W\text{ is compact, open and minimizing for }\alpha\right\}.
$$
\end{definition}
It is again clear that these are $\alpha$-invariant subgroups of $G$. That the iterated kernel of $\alpha$, that is, $\left\{ x\in G \mid \exists n \text{ such that }\alpha^n(x) = \ident\right\}$, need not be closed is seen in Example~\ref{ex:not_TB}. Hence the need to take the closure of this set in the definition of $\ikk{\alpha}$ in order to ensure that it is closed. As the intersection of closed subgroups, $\nub{\alpha}$ is closed.

\begin{proposition}
\label{prop:ik_nub_par}
Let $\alpha\in \End{G}$. Then 
$$
\ikk{\alpha} \leq \nub{\alpha} \leq \ml{\alpha}
$$
and each subgroup is $\alpha$-stable. Furthermore, $\nub{\alpha}$ and $\ikk{\alpha}$ are compact, $\ikk{\alpha}$ is normal in $\pbn{\alpha}$ and the endomorphism induced on $\pbn{\alpha}/\ikk{\alpha}$ is an automorphism. 
\end{proposition}
\begin{proof}
That $\ikk{\alpha} \leq \nub{\alpha}$ and is compact follows from Corollary~\ref{cor:bounded_ in_LV}, which implies that $\ikk{\alpha}\leq L_V$ for every subgroup $V$ that is tidy for $\alpha$. It is also immediate from the definition that $\ikk{\alpha}$ is $\alpha$-stable. 

Since $\nub{\alpha}$ is the intersection of compact groups, it is compact. Corollary~\ref{cor:tidy_subgroups} shows that in fact
$$
\nub{\alpha} = \bigcap\left\{ W_+\cap W_- \mid W\text{ is tidy for }\alpha\right\}.
$$  
Since $W_+\cap W_-$ is $\alpha$-stable, it follows immediately that $\nub{\alpha}$ is $\alpha$-invariant, and by a compactness argument that it is $\alpha$-stable.  
Therefore $\nub{\alpha}$ is contained in $\ml{\alpha}$. 

It follows in particular from the foregoing that $\ikk{\alpha}$ is a subgroup of $\pbn{\alpha}$. Let $x$ in $\ikk{\alpha}$ satisfy $\alpha^n(x) = \ident$ for some $n$. (Such elements are dense in $\ikk{\alpha}$.) Let $y$ be in $\pbn{\alpha}$ and have bounded $\alpha$-regressive sequence $\{y_n\}_{n\in\mathbb{N}}$. Then $\alpha^n(yxy^{-1}) = \ident$ and $\{y_nx_ny_n^{-1}\}_{n\in\mathbb{N}}$ is a bounded $\alpha$-regressive sequence for $yxy^{-1}$. Therefore $yxy^{-1}$ belongs to $\ikk{\alpha}$ and it follows that $\ikk{\alpha}$ is normal in $\pbn{\alpha}$.

The restriction of $\alpha$ to $\pbn{\alpha}$ is an endomorphism that is onto by definition of $\pbn{\alpha}$. The co-restriction to $\pbn{\alpha}/\ikk{\alpha}$ is therefore an endomorphism that is onto also. Consider $x\ikk{\alpha}$ in the kernel of this co-restriction, so that $x\in \pbn{\alpha}$ and $\alpha(x)\in \ikk{\alpha}$. Then, since the restriction of $\alpha$ to $\ikk{\alpha}$ is onto, there is $y\in \ikk{\alpha}$ such that $\alpha(y) = \alpha(x)$. Hence $y^{-1}x$ belongs to $\ikk{\alpha}$, and it follows that $x$ is in $\ikk{\alpha}$. Therefore the kernel of the co-restriction of $\alpha$ to $\pbn{\alpha}/\ikk{\alpha}$ is trivial and this co-restriction is an automorphism. 
\end{proof}
\begin{remark}
\label{rem:ik_nub_par}
A similar but stronger conclusion than that of Proposition~\ref{prop:ik_nub_par} is reached in~\cite[Theorem~A]{reid} under the stronger hypothesis that $G$ is a compact group that has only finitely many open subgroups of any given index. 
\end{remark}
The structure of $\nub{\alpha}$ may be described in some detail as a result of Proposition~\ref{prop:ik_nub_par}. The quotient $\nub{\alpha}/\ikk{\alpha}$ is the nub of an automorphism, and therefore has the structure given in~\cite{Wi:nub}. Since the kernel of the restriction of $\alpha$ to $\pbn{\alpha}$ is closed and contained in the compact set $\ikk{\alpha}$, we also have the following.
\begin{corollary}
\label{cor:ik_nub_par} The restriction of $\alpha$ to $\pbn{\alpha}$ has compact kernel and is therefore a proper map. \endproof
\end{corollary}

If $s(\alpha)>1$, then $\pbn{\alpha}$ is non-trivial and it follows that $\bigcap_{n\geq0}\alpha^n(G)$ is large. If $s(\alpha) = 1$, then this intersection may be trivial which, is the case most obviously if $\alpha$ itself is trivial. It might be wondered what can be said about $\bigcap_{n\geq0}\alpha^n(G)$ if it is supposed that $\bigcap_{n\geq0}\overline{\alpha^n(G)}$ is non-trivial. In the case when $G$ is metrizable and $\alpha$ is assumed to have dense range, $\bigcap_{n\geq0}{\alpha^n(G)}$ is dense in $G$ for purely topological reasons, see the Mittag-Leffler Theorem of Bourbaki~\cite[II.3.5]{Bourbaki_Top} or \cite{Runde}. However, the next example suggests that no stronger conclusion follows from $\alpha$ being an endomorphism. 
\begin{example}
\label{ex:rangestrange}
Let $G = ({F}_2((t)),+)$ be the additive group of the field of formal Laurent series over the field of order~$2$. Denote elements of $G$ by $f = \sum_{n\geq N} f_nt^n$, where $N\in\mathbb{Z}$, and define $\alpha\in \End{G}$ by
$$
\alpha(f)_n = \begin{cases}
f_{2n-2}, & \text{ if }n\leq0\\
f_{-n}, & \text{ if }n>0\text{ is odd}\\ 
f_{n/2}, & \text{ if }n\geq2\text{ is even}
\end{cases}.
$$
Then $\alpha^n(f) \to \ident$ as $n\to\infty$ for every $f\in G$ and $\bigcap_{n\geq0} \alpha^n(G)$ is the dense subgroup of Laurent polynomials, that is, the functions with finite support. 
\end{example}

\section{Examples}
\label{sec:examples}

The paper concludes by pointing out that totally disconnected, locally compact groups appearing in the literature do possess non-invertible endomorphisms that have been little studied. The first example is of a non-Hopfian group and the second of a non-co-Hopfian group. The methods introduced in this paper have the potential to be used for a full analysis of the endomorphisms of these and other groups. 

\begin{example}
The \emph{Baumslag-Solitar} groups $BS(m,n)$, for $m,n\in \mathbb{Z}\setminus \{0\}$ have the finite presentation $BS(m,n) = \langle a, t \mid t^{-1}a^mt = a^n\rangle$. It was shown in~\cite{BaumSol} that, provided that $m$ does not divide $n$ and $n$ does not divide $m$ and that $m$ and $n$ do not have the same prime divisors, then $BS(m,n)$ is non-Hopfian, that is, has an endomorphism that is onto but not one-to-one. The endomorphism exhibited is 
\begin{equation}
\label{eq:BSendo}
\alpha : t\mapsto t, \quad a\mapsto a^p,
\end{equation} 
where $p$ is a prime that divides $m$ but not $n$. 

Baumslag-Solitar groups embed densely into totally disconnected, locally compact groups, denoted $G_{m,n}$. To see this, note that the subgroup $\langle a\rangle$ is commensurated by $\BS(m,n)$, which implies that the closure of $\rho(BS(m,n))$ is a locally compact subgroup of $\Sym(\BS(m,n))/\langle a \rangle$, where $\rho : \BS(m,n) \to \Sym(\BS(m,n))/\langle a \rangle$ denotes the natural embedding, see~\cite{EldWi}. The embedding obtained by taking the closure of $BS(m,n)$ in the automorphism group of the Bass-Serre tree of the HNN-extension arising from the isomorphism $\langle a^m\rangle \to \langle a^n\rangle$ yields an isomorphic completion of $\BS(m,n)$. Under this completion, the closure of $\langle a\rangle$ is a compact, open subgroup isomorphic to $\bigoplus \mathbb{Z}_{p_i}$ where the sum is over the prime divisors, $p_i$, of $m/\gcd\{m,n\}$ and $n/\gcd\{m,n\}$. 

The endomorphism $\alpha$ of~\eqref{eq:BSendo} extends continuously to $G_{m,n}$ because its restriction to the open subgroup $\langle a\rangle$ is continuous. Since the extension to $G_{m,n}$ satisfies 
$$
\alpha\left(\bigoplus \mathbb{Z}_{p_i}\right) = p\left(\bigoplus \mathbb{Z}_{p_i}\right) < \bigoplus \mathbb{Z}_{p_i},
$$
$\bigoplus \mathbb{Z}_{p_i}$ is minimising and $s(\alpha) = 1$. Composing $\alpha$ with an inner automorphism $x\mapsto t^kxt^{-k}$ yields further endomorphisms. When $k>0$, $p$ divides the scale of the inner automorphism and composing with $\alpha$ reduces reduces the scale by a factor of $p$. When $k<0$, $p$ does not divide the scale of the inner automorphism and composing with $\alpha$ does not change the scale. \end{example}

\begin{example}
The \emph{hierarchomorphism group} of Neretin is the group $\text{Hier}(\mathcal{T})$ of homeomorphisms of the boundary $\partial\mathcal{T}$ of the tree $\mathcal{T}$ that extend to almost automorphisms of the tree, that is, to automorphisms of the forest obtained when a finite subtree is removed, see~\cite{Ner1,Ner2}. This group is compactly generated and abstractly simple when the tree is the regular tree, $\mathcal{T}_n$, with valency $n+1$, see~\cite{Kap} where it is called the \emph{spheromorphism group} and denoted~$N_n$.  The latter term and notation are used here. 

An endomorphism $\alpha : N_n\to N_n$ that is one-to-one but not onto may be defined as follows. When one vertex is deleted from $\mathcal{T}_n$ a forest, $\mathcal{F}$, of $n+1$ rooted trees in which every vertex has $n$ children remains. If the root of one of the trees in $\mathcal{F}$ is deleted, then a forest of $2n$ such rooted trees remains. Partition this forest as $\mathcal{A}\cup \mathcal{B}$ where $\mathcal{A}$ is a set of $n+1$ rooted trees and $\mathcal{B}$ a set of $n-1$ rooted trees. Fix an isomorphism $\phi : \mathcal{F}\cup\partial\mathcal{F} \to \mathcal{A}\cup \partial\mathcal{A}$. Every spheromorphism in $N_n$ may be realised as an automorphism of a sub-forest of $\mathcal{F}$ that extends to a homeomorphism of $\partial\mathcal{F}$. For each spheromorphism, $x$, define
$$
\alpha(x)(\omega) = \begin{cases}
\phi(x(\phi^{-1}(\omega))), & \text{ if }\omega\in\partial\mathcal{A}\\
\omega, & \text{ if }\omega\in\partial\mathcal{B}
\end{cases}.
$$
The $\alpha(x)$ is a spheromorphism in $N_n$ and $\alpha : N_n\to N_n$ is an endomorphism. The scale of $\alpha$ and it tidy subgroups will depend on the choice of $\phi$.

Since $N_n$ is simple, any non-trivial endomorphism has trivial kernel and so, in particular, $N_n$ is Hopfian.
\end{example}

\section*{Acknowledgements}
This research is supported by the Australian Research Council Discovery Grant DP0984342.

\end{document}